\documentclass[11pt, reqno]{amsart}
\usepackage{amsmath,amsthm,amssymb,bm}
\usepackage{graphicx, color,blindtext}
\usepackage{caption}
\usepackage{subcaption}
\usepackage[export]{adjustbox}
\usepackage[titletoc,toc]{appendix}

\usepackage{floatrow}
\usepackage{dsfont}
\usepackage{stmaryrd}
\usepackage{mathrsfs}
\usepackage{graphicx}
\usepackage{float}
\usepackage{algorithm}
\usepackage{algpseudocode}
\usepackage{tikz}
\usepackage{epstopdf}

\usepackage[colorlinks=true,pagebackref=true]{hyperref}
\hypersetup{urlcolor=blue, citecolor=red, linkcolor=blue}

\usepackage{ulem} 

\everymath{\displaystyle}

\theoremstyle{plain}
\newtheorem{theorem}{Theorem}
\newtheorem*{Rtheorem}{Theorem}

\newtheorem{lemma}[theorem]{Lemma}
\newtheorem{proposition}[theorem]{Proposition}
\theoremstyle{definition}

\theoremstyle{remark}
\newtheorem{remark}[theorem]{Remark}
\newenvironment{proof-sketch}{\noindent{\bf Sketch of Proof}\hspace*{1em}}{\qed\bigskip}

\newcommand{\RR}{\mathbb{R}}

\newcommand{\nn}{\mathbb{N}}

\newcommand{\dd}{\mathrm{d}}

\newcommand{\Om}{\Omega}

\DeclareMathOperator*{\argmax}{arg\,max}
\DeclareMathOperator*{\argmin}{arg\,min}
\newcommand{\odiv}{\operatorname{div}}

\newcommand{\K}{\mathcal K}
\newcommand{\F}{\mathcal F}

\newcommand{\CC}{\mathcal C}

\newcommand{\Wass}{\mathds{W}}
\newcommand{\kk}{\mathsf{k}}

\newcommand{\G}{\mathcal G}
\newcommand{\N}{\mathcal N}
\newcommand{\Lip}{\textup{Lip}}
\newcommand{\M}{\textup{(M)}}
\newcommand{\D}{\mathcal D}
\newcommand{\dist}{\mathbf{D}}

\newcommand{\A}{\mathcal A}
\newcommand{\B}{\mathcal B}
\newcommand{\id}{\mathrm{id}}

\newcommand{\etal}{\textit{et al}., }
\newcommand{\ie}{i.e.,\,}
\newcommand{\eg}{e.g.,\,}
\newcommand{\cf}{c.f.\,}

\newcommand{\prox}{\textbf{Prox}}
\newcommand{\proj}{\textbf{Proj}}
\newcommand{\p}{\textbf{p}}

\newcommand{\x}{\textbf{x}}

\newcommand{\n}{\textbf{n}}

\newcommand{\bnu}{\boldsymbol{\nu}}

\newcommand\restr[2]{{
		\left.\kern-\nulldelimiterspace 
		#1 
		\vphantom{\big|} 
		\right|_{#2} 
}}

\def\1{\raisebox{2pt}{\rm{$\chi$}}}
\newcommand{\res}        {\!\!\mathop{\hbox{
			\vrule height 7pt width .5pt depth 0pt
			\vrule height .5pt width 6pt depth 0pt}}
	\nolimits}

\allowdisplaybreaks

\usepackage{mathtools}
\mathtoolsset{showonlyrefs}
\DeclareMathOperator{\dive}{div}

\newcommand{\I}{\mathbb{I}}
\newcommand{\R}{\mathbb{R}}

\usepackage{stackengine}
\usepackage{scalerel}
\usepackage{xcolor,amssymb}

\allowdisplaybreaks
\newcommand{\cs}{$^\dagger$} \newcommand{\cm}{$^\ddagger$}

\title{Prediction-Correction Pedestrian Flow by Means of Minimum Flow Problem}

\author[H.  Ennaji]{Hamza Ennaji \cs}\thanks{\cs Normandie Univ, UNICAEN, ENSICAEN, CNRS, GREYC, France. Email : hamza.ennaji@unicaen.fr}

\author[N. Igbida]{Noureddine Igbida\cm}
\author[G. Jradi]{Ghadir Jradi\cm}\thanks{\cm Institut de recherche XLIM-DMI, UMR-CNRS 6172, Facult\'e des Sciences et Techniques, Universit\'e de Limoges, France. Emails:   noureddine.igbida@unilim.fr,  ghadir.jradi@unilim.fr }

\date{\today}

\begin{document}
	
	\begin{abstract}
		We study a new variant of  mathematical prediction-correction model for crowd motion. The prediction phase is handled by a transport equation where the vector field is computed via an eikonal equation $\Vert \nabla\varphi\Vert=f$, with a positive continuous function $f$ connected to the speed of the spontaneous travel.  The correction phase is handled by a new version of the minimum flow problem. This model is flexible and can take into account different types of interactions between the agents, from gradient flow in Wassersetin space to granular type dynamics like in sandpile.  	Furthermore, different boundary conditions can be used, such as non-homogeneous Dirichlet (\eg outings with different exit-cost penalty) and Neumann boundary conditions (\eg  entrances with different rates).  Combining finite volume method for the transport equation and Chambolle-Pock's primal dual algorithm for the eikonal equation and minimum flow problem, we present numerical simulations to demonstrate the behavior in different scenarios.
	\end{abstract}
	
	\maketitle

\section{Introduction}

Macroscopic model for a congested   pedestrian flow involves treating the crowd as a whole and  is applicable  for  large crowds. It was first introduced in \cite{Bord}  and developed in \cite{Helbing1,Helbing2}. In these models, the crowd  behave similarly to a  moving fluid in a spatio-temporal dynamic governed by a  flow velocity  vector field $U$.  
Thus the master equation of  each macroscopic crowd flows model is the continuity equation:
\begin{equation}\label{transport1}
	\partial_t \rho +\dive(\rho\: U) =0,
\end{equation}
where $ \rho=\rho(t,x) $  the density of the individuals, at time $t\geq 0$ and at the position $x\in \RR^N$ ($N=2$), needs to
accurate some admissible global distribution of the population. Though there is much speculation, discussion and experience to define  appropriate  choice of  flow velocity  vector field $U,$
there is no  definitive universal choice to describe  crowed motion in general.
The main difficulties lies in the fact that while maintaining a suitable dynamic esteeming  the admissible global distribution $\rho$, $U$ needs to manage both, the overall behavior of the crowd (for example of reaching an objective like exit, point of interest, avoidance of danger, etc.) and certain local behavior of pedestrians (pedestrian in a hurry, pedestrian who adapts their speed, pedestrian who avoids the crowd, pedestrian attracted by the crowd, etc).

Inspired by traffic flow models, many crowd motion models were performed essentially in one-dimensional space  (\cf \cite{Colombo&Rosini,Helbing1,Helbing2}).  In higher dimensions, Bellomo and Dogbe (\cf \cite{Bellomo&Dogbe,Dogbe}) proposed coupling the continuity equation with 
$$\partial_t U +(U\cdot \nabla_x)U =F(\rho,U),   $$
where the   motion is governed by $F$, which has two parts: a relaxation term towards a definite speed, and a repulsive term to take into account that pedestrians tend to avoid high-density areas. A barrier method was proposed by Degond (cf. \cite{Degond&al}) wherein the motion $F$ depends on a pressure that blows up when the density approaches a given congestion density.  Piccoli and Tosin proposed another class of models in the framework of a time-evolving measure in \cite{Piccoli&Tosin1,Piccoli&Tosin2}.  In their model
the velocity of the pedestrian is composed by two terms: a desired velocity and an	interaction velocity.

Roger Hughes proposed a completely different approach to describing pedestrian dynamics in \cite{Hughes1}, where a group of people wants to leave a domain with one or more exits/doors as quickly as possible. His main idea was to include some kind of saturation effects in the vector field.  He considered   $U=U[\rho]$ driven by the gradient of  a potential $\Phi$ and weighted by a nonlinear mobility  $f=f(\rho)$. More precisely
\[
U=f(\rho)^2\nabla \Phi \quad \hbox{ and } \Vert \nabla \Phi\Vert =1/f(\rho),
\]
where mobility includes saturation effects, \ie degenerate behavior when approaching a given maximum density  $\rho_{max}$ (assumed to be known); for instance one can take $f(\rho)=(\rho-\rho_{max})^2$ among others.   See also \cite{Coscia&Canavesio} and \cite{Hughes2} for  further details.

To handle  local behaviors of pedestrian,   we go here with second order PDE for crowed motion  to perform congestion phenomena which may appear  if one consider velocity field  $U$    looking out solely to the exists (doors).  The main idea is to get in  $U$ together a   vector field  $V$ with an overview   looking out   to the exit
and some kind of patch  $W$, a vector filed with a local view  looking out to the allowable neighbor positions   taking into account the local distribution of the pedestrian.  To come out with  $U$ through this perspective, we process by splitting the dynamic into two instantaneous phases: a first one, the so-called  prediction phase, where the pedestrians   move 	along the given  vector filed $V,$  the so called spontaneous velocity field,  and a second phase, the correction,   which generates a patch  $W$  that enables the pedestrian to move 	along allowable local paths to avoid congestion and maintain  admissible global distribution of the pedestrian.  	A typical example of this point of view remains to be   the 	constrained diffusion-transport equation  which was performed   in the pioneering work  by B.  Maury and al.   (cf.  \cite{MRS1}) through a predicting-correcting algorithm using a gradient flow in the Wasserstein space of probability measures.   In this paper, we use a new manner  to  handle this perspective.  In contrast with \cite{MRS1} where the author straighten up the density using some kind of projection in $\Wass_2-$Wasserstein space in the correction phase,  our approach   is based on a new version of minimum flow problem.   The approach  is  flexible and makes it possible to integrate several scenarios to deal with congestion. One can see also \cite{IgCrowd} where the approach is used to  study similar  dynamic in the case of two populations. In particular it allows to retrieve and compute otherwise  the typical  model of B. Maury, where  the patch $W=W[\rho]$ is traced strictly in the so called congested/saturated  regions as follow
\begin{equation}
	W[\rho] =- \nabla p, \hbox{ with } p\geq 0 \hbox{ and } p(\rho-1)=0.
\end{equation} 	
Here $\rho\equiv 1$ workouts the utmost distribution of the population in $\Omega.$ Therefore,  via this approach, the   proposed system
 reads \begin{equation}\label{Maury1}
	\left\{
	\begin{array}{ll}
		\displaystyle \frac{\partial \rho }{\partial t}  +\dive(  \rho  \: (V- \nabla p)   )=  0\\  \\
		p\geq 0,\: 0\leq \rho\leq 1,\: p(\rho-1)=0.
	\end{array}
	\right.
\end{equation}
Furthermore, the  approach  enables  to built   a new model based on granular dynamic like  in  sandpile, presuming that  individuals behave like grains in the congested zones.   	In some sense, at the microscopic level,  the individuals travel    by accruing randomly  to the crowed, being placed either upon a heretofore unoccupied position in the direction of the exit  or else upon the top of the stack of the crowd.  Moreover, the local movement of the  individuals may be weighted by a given function $\kk$ connected to the speed of the  spontaneous local movement.  In this case,  we prove that the patch is given by 
\begin{equation}
	W[\rho] =- m\nabla p 
\end{equation} 
with unknown $m$ and $p$ satisfying  
\begin{equation}
	m\geq 0,\: p\geq 0,\: \vert \nabla p\vert \leq \kk,\: p(\rho-1)\hbox{ and } m\:(\vert \nabla p\vert -\kk)=0 .
\end{equation}  	 
Here  $m\geq 0$ is Lagrange multiplayer associated with the additional constraint $\vert \nabla p\vert \leq \kk.$ 	 The approach enables also to handle and integrate  different boundary conditions.  Neumann boundary condition is   connected to the crossing  boundary amount, and Dirichlet  is connected to the possibility of crossing some parts of the boundary with different charges.

After all,  via this approach, we introduce a  new model  of granular type : 
\begin{equation} \label{evolgran0}
	\left\{
	\begin{array}{ll}
		\displaystyle \frac{\partial \rho }{\partial t}  +\dive(  \rho  \: (V- m \nabla p)   )=  0\\  \\
		0\leq \rho\leq 1,\: 	p\geq 0,\:  \vert \nabla p\vert \leq \kk \\  \\ 
		p(\rho-1)=0,\:  m(\vert \nabla p\vert -\kk)=0,
	\end{array}
	\right.
\end{equation}
subject to mixed boundary conditions (not necessary homogeneous), to describe a crowd motion where the movement of the agent is of   granular type like in sandpile.    In this paper, 
we   propose  its numerical study based on  a new manner  to  handle the predicting-correcting algorithm to build the patch $W$.  Over and above the transport equation \eqref{transport1},  we  proceed using as well a new version of minimum flow problem  for optimal assignation  as a step in the process to find the right assignment of the pedestrian.  Roughly speaking, in the correction step we  put together tow nested optimization  procedures:  a computation of   a  minimum flow  with  gainful assignment  towards a specific part of  the boundary  (towards the exit) for arbitrary target, and then a coming up with the right target among all admissible ones.   We show how one can retrieve and compute otherwise  the typical  model of B.Maury \etal  (\cf  \cite{MRS1})  that we call up above. Then, we focus on the new model based on granular dynamics-like for sandpile.

The theoretical study of \eqref{evolgran0} is a challenging problem, especially   existence and uniqueness questions, that we'll treat likely in forthcoming works.  Recall that, the case where the PDE is of diffusive type like in \eqref{Maury1}, the model is very employed to describe the behavior of population subject to global behavior governed by a vector field $V$    and a local one governed by the patch $W[\rho] $  (\cf \cite{MRS1,MRS2,MRSV,MS} and the references therein). The uniqueness of a solution is a   hard issue for these kind of problems that was treated recently by the second author  in \cite{Igshuniq} (see also \cite{DM,Noemi&Markus}).

\subsection*{Organization of  the paper.} 
This paper is organized as follows. In Section-\ref{section:model} we present our model, we give the details of each of its steps and we discuss two peculiar  related PDEs to this model as well as some duality results on which our algorithm reposes. In Section-\ref{section:numerics}, we show how to discretize the model. 	 Since the approximation of the continuity equation is more or less classical, the novelty will be the use of a primal-dual method to solve the Beckmann-like problem.  More particularly, this is given in Algorithm-\ref{alg:pd}. In Section-\ref{section:examples} we given several examples to illustrate our approach and we compare with some related works. Finally, we recall some tools and give some technical proofs in the Appendix.

\section{The model }\label{section:model}

We  consider  an exit scenario, where $\Omega\subset\RR^N$ ($N=2$)  is a bounded   open set with regular boundary $\displaystyle \partial\Omega =\Gamma_N\cup \Gamma_D.$
The set $\Omega$ represents    the region where the crowd is moving, $\Gamma_N$  represents the (impenetrable) walls and $\Gamma_D$ the exits/doors.  

\subsection{Minimum flow problem}

The key idea concerning  the  minimum flow problem goes back to Beckmann \cite{Beckmann}.  It consists in finding the optimal  traffic flow field $\Phi$  between the two distributions  given by $\mu_1$ and $\mu_2.$ That is to find  the vector field $\Phi$ which satisfies  
the divergence equation
\begin{equation}\label{baleq1}
	- \dive(\Phi)  =\mu_1-\mu_2 \hbox{ in }  \overline \Omega,
\end{equation}
and minimize a total cost of   the traffic $\int F(x, \Phi(x) )\:\dd x,$ where    
$F\: :\:  \Omega\times \RR^N\to \RR^+$ is a given function assumed to be at least continuous and convex with respect to the second variable.  
The  equation \eqref{baleq1} needs to be understood in the sense of $\D'(\overline \Omega).$ In particular, the equation assigns a fixed normal trace to $\Phi$ on $\partial \Omega$ which is connected to the formal values of $\mu_1-\mu_2$  on $\partial \Omega.$  

Here, we use a new variant  to handle the pedestrian flow  and carried out the   patch $W$ for the  spontaneous velocity field  when the pedestrian is  hindered by the other one.   Indeed, we   work with a modified  traffic cost which   handles   some kind of   gainful assignment  towards a specific part of  the boundary $\Gamma_D.$  More precisely, we   
consider the following momentum cost  of the traffic
\[
\mathcal M (\Phi):=  \int_\Omega F(x,  \Phi  (x) )\: \dd x -  \int_{\Gamma_D} g(x)\: \Phi\cdot \bnu\:\dd x ,
\]
where $\Phi\cdot\bnu$ denotes the normal trace of $\Phi$  and   $g$ patterns a given  gainful charge for the  assignment  towards  $\Gamma_D.$       Of course, for the optimization problem we have in mind we need to keep unrestricted  the normal trace of $\Phi$ on $\Gamma_D$.   Thus,   the balance equation \eqref{baleq1}  turns into 
\begin{equation}
	\label{baleq2}
	- \dive(\Phi)  =\mu_1-\mu_2    \hbox{ in }    \D'(\overline \Omega\setminus \Gamma_D) .  
\end{equation}
See here, that the normal trace of $\Phi$ on the odd  part  $\Gamma_N$ remains to be  given by   $\mu_1-\mu_2$ on $\Gamma_N.$ For instance, working with $\mu_1$ and $\mu_2$ supported in $\Omega,$  we  keep unrestricted  the normal trace of $\Phi$ on $\Gamma_D$ but assigned it to $0$ on $\Gamma_N.$

This being said, we consider the transportation cost associated with  given densities $\mu_1$ and $\mu_2$  to be
\begin{equation}\label{cfmu}
	\begin{array}{l}
		\inf_{\Phi}\Big \{  \int_\Omega F(x,  \Phi  (x) )\: \dd x -  \int_{\Gamma_D} g(x)\: \Phi\cdot \bnu \:     \dd x \: :\:    \:      - \dive(\Phi)  =\mu_1-\mu_2    \hbox{ in }    \D'(\overline \Omega\setminus \Gamma_D)     \Big\}.
	\end{array}
\end{equation}
Actually,  for any arbitrary distributions $\mu_1$ and $\mu_2,$  the optimization problem   \eqref{cfmu} aims to minimize  both  the transportation between $\mu_1$ and $\mu_2,$ in $\Omega$ and towards $\Gamma_D,$   by means of the cost function $F $ in $\Omega$, as well as the transportation towards the boundary $\Gamma_D$ paying the gainful charge $g(x)$ for each target position  $x\in \Gamma_D,$ respectively.    Moreover,  the new formulation enables to handle as well    a provided incoming (or outgoing) flux on the remaining part $\Gamma_N.$ 

Notice here, that one needs to be careful   with the notion of trace of $\Phi$ on the boundary since it is not well defined for all $\Phi$. 
One needs to be careful here with the notion of trace of $\Phi$ on the boundary since it is not well defined for all $\Phi$. However, working in
\[
H_{\mathrm{div}} :=\Big\{ \Phi\in L^2(\Omega):~ \dive(\Phi)\in L^2(\Omega) \Big\},
\]
enables us to define $\Phi\cdot \bnu$ on $\Gamma_D$  in the right sense. Indeed, let  $\gamma_{0}: H^{1}(\Omega)\to L^{2}(\Gamma)$   be the linear and continuous mapping  satisfying  $\gamma_{0}(u) = u_{|\Gamma}$ for all $u\in  C(\overline \Omega)$, where $\Gamma=\partial \Omega.$ Then, defining $H^{1/2}(\Gamma) = \gamma_{0}(H^{1}(\overline \Omega))$ and $H^{-1/2}(\Gamma)$ its dual, there exists a continuous trace operator $\gamma_{\n}: H_{\mathrm{div}}(\Omega) \to H^{-1/2}(\Gamma)$ such that $\gamma_{\bnu}(\Phi) = \Phi\cdot\bnu$ for any $\Phi\in \D(\overline \Omega )^N$.  Thanks to Gauss's Theorem, we have
$$\langle  \gamma_{\bnu}(\Phi),\gamma_{0}(u)\rangle_{H^{-1/2},H^{1/2}}= \int_{\Omega}\Phi\cdot \nabla u \: \dd x +\int_{\Omega} u \: \dive(\Phi)\:  \dd x \:\: \text{ for all }\:\:  u\in H^1(\Omega), \Phi\in H_{\mathrm{div}}(\Omega).
$$
To simplify the presentation, we denote $\Phi\cdot\bnu :=  \gamma_{\bnu}(\Phi),$ and moreover
\[
\int_{\Gamma_D} g(x)\: \Phi\cdot\bnu\:  \dd x := \langle \gamma_{\bnu}(\Phi), \tilde g\rangle_{H^{-1/2},H^{1/2}},
\]
where $\tilde g\chi_{\Gamma_D} =g $ for $\tilde g\in H^{1/2}.$  Yet, one needs to assume that such $\tilde g$ exists (see the assumptions in Section \ref{subsection:pde}).

\bigskip 
Before ending this section,  let us recall that a  similar problem to \eqref{cfmu} appears in  \cite{EINHJ}  in the study of Hamilton-Jacobi equation (see also \cite{EINsfs} and \cite{EINfinsler}). It appears also on a different form  in the study of some Sobolev regularity for degenerate elliptic PDEs in \cite{Sanatambrogio-duality}.   Indeed, to avoid   technicalities related to the normal trace of the flux on the boundary,  it  is possible to rewrite  \eqref{cfmu} as follows 
\begin{equation}   \label{cfmuS}
	\begin{array}{c}
		\inf_{\Phi,\upsilon}\left \{  \int_\Omega F(x,  \Phi  (x) )\: \dd x -\int_{\Gamma_D} g(x)\: \upsilon (x) \:     \dd x   \: :\:    \:   	- \dive(\Phi)  =\mu_1-\mu_2    \hbox{ in }    \D'(\overline \Omega\setminus \Gamma_D)  \right.\\  \\ 
		\hbox{ and } \Phi\cdot \bnu =\upsilon \hbox{ on } \Gamma_D   \Big\}.
	\end{array}
\end{equation}

\begin{remark}
	Taking non-homogeneous  boundary data $g$ and $\eta$ enables the treatment of  congestion crowd motion in an urban area with many issues :  incoming issues included in $\Gamma_N$ with  supply  rate given by $\eta$ and outgoing issues in $\Gamma_D$ with some kind of rate of return  pictured  by $g.$ 
	
\end{remark}

	In \cite{Sanatambrogio-duality}, the author studied some particular cases of \eqref{cfmuS} (like for instance the case where $\Gamma_D=\emptyset$ and also $\Gamma_D=\partial \Omega$).    We notice also that in \eqref{cfmu}, the infimum in $\Phi$ is not reached in general and one looks in some situations (like for  homogeneous $F$) for measure flow fields instead. This is related to the question of regularity of the transport density in mass transportation  (see for instance the recent paper \cite{Dweik}   and the references therein).

	\subsection{The algorithm} \label{Salgorithm}

	The main idea of prediction-correction algorithm  is  to split the dynamic into instantaneous successive processes : prediction then correction. The prediction  step aims beforehand    to move   the population  through a spontaneous velocity field.  For this to happen,  we use simply the transport equation \eqref{transport1} with $U=V,$ where $V$ derives from a potential  governed by  fast exit access trajectories.  Afterward, as though the output of the prediction  may  be not feasible, we propose to catch up  the upright deployment by applying  the minimum flow  assignment  process \eqref{cfmu} to the output of the prediction step; that we denote for the moment by $\tilde \rho$ and which should be a priori  an $L^\infty$ function.   
	Moreover, assuming  that the dynamic is subject  to some  supply of population through   incoming issues included in $\Gamma_N,$ we propose to take 
	$$\mu_1= \tilde \rho \res \Omega + \eta \res \Gamma_N,$$  
	where $\eta$ precisely designates the  incoming  supply through $\Gamma_N$. In this case, the constraint 
	$	 -\dive(\Phi)  =\mu_1-\rho    \hbox{ in }   \D'(\overline \Omega\setminus \Gamma_D)  $ is equivalent  to say 
	$$ -\dive(\Phi)  =\tilde \rho -\rho  \quad  \hbox{ in }   \D'(  \Omega )  \quad\hbox{ and } \quad \Phi\cdot \bnu =\eta\quad \hbox{ on }\Gamma_N. $$
	The correction step we propose to construct $\rho$ requires to solve precisely the   following optimization problem 
	\begin{equation}\label{minproc}
		\begin{array}{c}
			\inf_{(\Phi,\rho)}\left \{  \int_\Omega F(x,  \Phi  (x) )\: \dd x -  \int_{\Gamma_D} g(x)\: \Phi\cdot\bnu \:  \dd x   \: :\:    0\leq \rho\leq 1,\:   -\dive(\Phi)  =\tilde \rho -\rho   \hbox{ in }   \D'(  \Omega )   \right.  \\   \\  \hspace*{1cm}       \hbox { and }   \Phi\cdot \bnu =\eta  \hbox{ on }\Gamma_N        \Big\}.
		\end{array}
	\end{equation}
	The right space for each terms in \eqref{minproc} will be given in the following section. See that the output of  the correction step provides as well   the correction associated with  $\tilde \rho$ and the suitable  flow for    the adjustment  of the dynamic.  We will see in the following section  how the optimal flux $\Phi$ enables to  carry  out  the patch  $W$  for the  spontaneous velocity field  when this is necessary.  
	
	\medskip
	So, the algorithm may be written as follows : we consider $T>0$ a given time horizon. For a given time step $\tau >0,$ we consider a uniform partition   of  $[0,T]$ given by
	$t_k=k\tau,$  $k=0,\dotsc,n-1.$ Supposing that we know the density of the population $\rho_k$ at a given  step $k,$  starting by $\rho_0.$ Then, we superimpose successively the following two steps :
	
	\begin{itemize}
		\item   \textbf{Prediction:}
		In this predictive step, the density of population trends to grow up into
		$$ \rho_{k+\frac{1}{2}} =\rho(t_{k+\frac{1}{2}}) ,$$
		where $\rho$ is the solution of the transport equation 
		\begin{equation}
			\label{eq:continuity} 
			\partial_t \rho+\dive(\rho\: V )=0\quad \hbox{ in }[t_k,t_{k+\frac{1}{2}}[,
		\end{equation}
		with  $ \rho(t_k)=\rho_k.$   Here, $V$ is spontaneous velocity field  given by the geodesics toward the  exit.  To built its  corresponding potential $\varphi$, we propose to solve the eikonal equation
		\begin{equation}\label{hjk}
			\left\{ \begin{array}{ll}
				\Vert \nabla \varphi\Vert =f\quad & \hbox{ in }\Omega,\\   \\
				\varphi =0 & \hbox{ on }\Gamma_D,
			\end{array}
			\right.
		\end{equation}
		where $f$ is a given positive continuous function. Then, the spontaneous velocity field  $V$ is given   by
		$V:=-\nabla \varphi/\Vert \nabla \varphi \Vert.$   One sees here that the solution of \eqref{hjk} (in the sense of viscosity) gives  the speedy   path  to the exit $\Gamma_D$. 	The potential $\varphi$ corresponds to the expected travel time to maneuver towards an exit. In particular, $\varphi$ is proportional to $f$ which may template space movement of traffic . As we will see, we can upgrade   the spontaneous velocity field   by taking $f$ depending on  the density on real time (like in Hugue's model).

		\item  \textbf{Correction:} In general it is not expected that $\rho_{k+\frac{1}{2}}$ to be an allowable  density of pedestrian, since this value may evolve outside the interval $[0,1].$
		We propose then to proceed by the Beckmann-like process we introduced above to find the right apportionment of the pedestrian. That is to find $\rho_{k+1}$ using the optimization problem  \eqref{minproc}.   More precisely, we propose to consider $\rho_{k+1}$ given by the following optimization problem
		\begin{equation} \label{argmink}
			\begin{array}{l}
				\argmin_{ \rho } \inf_{\Phi}\left \{    \int_\Omega F(x,  \Phi  (x) )\: \dd x -  \int_{\Gamma_D} g(x)\: \Phi\cdot\bnu \:  \dd x   \: :\: \rho \in L^\infty(\Omega) ,\:  0\leq \rho\leq 1,\:  \right.  \\   \\  \hspace*{1cm}  \Phi\in L^2(\Omega)^N,\:     -\tau\: \dive(\Phi)  =\rho_{k+1/2} - \rho  \hbox{ in }   \D'(  \Omega )    \hbox { and }  \Phi\cdot \bnu = \eta  \hbox{ on }\Gamma_N        \Big\}.
			\end{array}
		\end{equation}

	\end{itemize}

	\subsection{Related  PDE}\label{subsection:pde}
	
	The application    $F:\Omega\times \RR^N \to [0,\infty)$   is assumed firstly to be    continuous.   As a primer practical case,  one can consider the quadratic case, \ie
	$$F(x,\xi)= \frac{1}{2}\vert \xi \vert^2,\quad \hbox{  for any  }x\in \Omega\hbox{ and }\xi\in \RR^N.$$ 
	More sophisticated situations  arise by  considering non-homogeneous $F$  that weights the cost according to  space variables; like for instance 
	\begin{equation}\label{Fs}
		F(x,\xi)= \frac{c(x)}{s} \:  \vert \xi\vert^{s},\quad \hbox{  for any  }x\in \Omega\hbox{ and }\xi\in \RR^N,
	\end{equation} 
	with   $1< s<\infty.$  In particular,  with the parameter $c$  one can scale the cost by focusing on and/or avoiding certain regions in space.  	A formal computation using duality à la Fenchel-Rockafellar (see \eg \cite{Ekeland&Temam}) implies that  the infimum in \eqref{minproc} should coincide with 
	\begin{equation}
		\inf_{p}\left\{ \int p^+(x)\: \dd x+ \frac{1}{s'} \int  c(x)^{s'-1} \:  \vert  \nabla p(x) \vert^{s'} \: \dd x - \int  p(x)\:  \tilde \rho(x)  \: \dd x \: :\:  p\in W^{1,s'}(\Omega) ,\ p=g\hbox{ on }\Gamma_D   \right\},
	\end{equation}
	where $s'$ is the conjugate index of $s$, \ie it satisfies $\frac{1}{s}+\frac{1}{s'} = 1$.
	Moreover, $p$ and $(\rho,\Phi)$  are solutions of  both problems  respectively, if and only if $(p,\rho,\Phi)$ is a solution of the following PDE
	\begin{equation}
		\label{system2}
		\left\{
		\begin{array}{ll}
			\left. 	\begin{array}{l}
				\rho - \dive( \Phi) = \tilde{\rho}  \\   \\
				\rho \in \hbox{Sign}^+(p),\quad k\: \Phi = c^{s'}  \:  \vert \nabla p\vert^{s'-2} \nabla p  
			\end{array} \right\}  \quad & \hbox{ in } \Omega,\\  \\
			\Phi \cdot \bnu =\eta & \hbox{ on }\Gamma_N,\\  \\
			p=g    & \hbox{ on }\Gamma_D,
		\end{array}
		\right.
	\end{equation}
	where $\hbox{Sign}^{+}$ is  the maximal monotone graph given by
	$$
	\operatorname{Sign}^{+}(r)=\left\{\begin{array}{ll}
		1 & \text { for } r>0 \\
		{[0,1]} & \text { for } r=0 \\
		0 & \text { for } r<0
	\end{array} \quad \text { for } r \in\R.\right.
	$$
	In other words, $\rho\in\operatorname{Sign}^{+}(p)$ is equivalent to says that $0\leq \rho\leq 1$ and $p(1-\rho)=0$ in $\Omega.$  In this paper, we focus on the case where $s=1.$ For the treatment of the other cases, one can see \cite{IgCrowd} for more details.   In particular, one sees that the quadratic case is closely connected to  the system \eqref{Maury1} which was proposed by  B.  Maury \etal \cite{MRS1} in the framework of gradient flows in the Wasserstein space of probability measures.    As to the case \eqref{Fs},   dynamical model which comes off following  our approach is given by some kind of  non linear $s'-$Laplace equation 
	\begin{equation}\label{modelsprime}
		\left\{
		\begin{array}{ll}
			\displaystyle \frac{\partial \rho }{\partial t}  +\dive(  \rho  \: (V- W)   )=  0,\quad 	k\: W = c^{s'}  \:    \vert \nabla p\vert^{s'-2} \nabla p\\  \\
			p\geq 0,\: 0\leq \rho\leq 1,\: p(\rho-1)=0
		\end{array}
		\right\} \quad \hbox{ in } (0,\infty)\times \Omega 
	\end{equation}
	subject to boundary conditions 
	\begin{equation}
		\label{mixedBC}
		\left\{
		\begin{array}{ll}
			\Phi \cdot \bnu =\eta \quad & \hbox{ on }(0,\infty)\times \Gamma_N,\\  \\
			p=g    & \hbox{ on }(0,\infty)\times \Gamma_D.
		\end{array}
		\right.
	\end{equation}
	Notice that we use in \eqref{modelsprime}, the fact that  $\nabla p=\rho\: \nabla p,$ which is connected to $\rho\in  \operatorname{Sign}^{+}(p)$.

	As we said above,  we focus here on the case where 
	\begin{equation}\label{grancase}
		F(x,  \xi)=\kk(x)\: \vert \xi\vert  ,\quad \hbox{ for any }(x,\xi)\in \Omega\times \RR^N ,
	\end{equation}  
	where $0\leq \kk\in \mathcal C(\overline \Omega).$   	This  case  
	is closely connected to granular dynamic like sandpile  (see \cite{IgDu} and the references therein).   In other words the individuals behaves like grains of sand (see \cite{EvReza}  and also \cite{Igstoch} for a stochastic microscopic description of the granular dynamic),  in the congestion zone and not like a fluid as follows from the quadratic case.  A peculiar choice may be the same function $f$  
	
	Moreover, in connection  with gradient flow in the Wasserstein space, it is possible to connect our approach (in the case \eqref{grancase}) to the gradient flow in the  Wasserstein space of probability measures equipped with $\Wass_1.$  Indeed, in the case where $V\equiv 0,$ the link is well established at least in some particular case of nonlinearity  connecting $\rho$ to $p$   (cf.  \cite{AgCaIg}).

	\bigskip 
	
	To treat the problem \eqref{minproc},   we assume  that the boundary data $g$ and $\eta$ are such that 
	\begin{itemize}
		\item[(H1): ]  the exists $\tilde g\in  W^{1,\infty}(\Omega),$ such that  
		\begin{equation}\label{dominF*}
			\nabla \tilde g(x)\in	G(x):= \Big\{ \xi\in\RR^N\: :\:   \vert \xi\vert \leq \kk(x)\Big\},\hbox{ a.e. in }\Omega.   
		\end{equation}   and 
		\begin{equation}\label{condg}
			\tilde{g}=g \: \chi_{\Gamma_D}\quad \hbox{ on }\partial\Omega,
		\end{equation}  
		
		\item[(H2): ]
		there   exists $1<s<\infty$   such that 
		$$\eta \in W^{-\frac{1}{s},s}(\Gamma_N).$$   
	\end{itemize}

	Then, for any $\mu\in L^s(\Omega),$  we define 
	$$ \F (\mu):=\Big\{  \Phi\in L^{s}(\Omega)^N \: :\:  -  \dive(\Phi)  = \mu     \hbox{ in }      \Omega    \hbox{ and }    \Phi\cdot \bnu =\eta \hbox{ on } \Gamma_N    \Big\}.  $$ 
	Remind here, that  $ - \dive(\Phi)  = \mu$ in    $  \Omega $ and $  \Phi\cdot \bnu =\eta $ on  $ \Gamma_N $ needs to be understood in the sense that 	
	\begin{equation}
		\int_\Omega \Phi\cdot \nabla \xi \: \dd x = \int_{\Omega }\mu\: \xi\: \dd x + \int_{\Gamma_N} \eta\: \xi\: \dd x, \quad \hbox{ for any }\xi\in W^{1,s'}_{\Gamma_D } (\Omega).
	\end{equation}
	We are interested into the interpretation in terms of PDE of the solution of the problem 
	\begin{equation}  
		\N (\tilde\rho) := 	\inf_{\Phi,\rho}\left \{ \tau\:  \int_\Omega \kk(x)\: \vert \Phi(x)\vert\: \dd x -\tau\:  \int_{\Gamma_D} g\: \Phi\cdot \bnu  \:     \dd x   \: :\: \tau\:  \Phi\in  \F (\tilde \rho-\rho) \hbox{ and }\rho\in \K_1 \right\}, 
	\end{equation} 
	where  $0\leq \tilde\rho\in L^s(\Omega)$ is fixed  and $K_1$ is the set of admissible densities:
	\[
	\K_1 = \{\rho\in L^\infty(\Omega)\: :\: 0\leq \rho\leq 1\hbox{ a.e. in }\Omega  \}.
	\] 
	See that, all the terms in $\N (\tilde\rho) $ are well defined. Indeed, since   $\Phi\in L^s(\Omega)^N$ and $\nabla \cdot \Phi \in L^s(\Omega),$  the normal trace of $\Phi$ is well defined on $\Gamma_D$ and $\Gamma_N.$    Actually $\int_{\Gamma_D} g\: \Phi\cdot \bnu  \:     \dd x $ and $  \Phi\cdot \bnu =\eta \hbox{ on } \Gamma_N$  need  to be understood, respectively,  in the sense of 
	$$   \langle \Phi \cdot \bnu,\tilde g \rangle_{ W^{-1/s,s}(  \Gamma_D),  W^{1-1/s',s'}(  \Gamma)}  $$ 
	and 
	$$ \int_\Omega\Phi \cdot \nabla   \xi  \: \dd x +  \int_\Omega   \xi \: \nabla \cdot \Phi\: \dd x = \langle \eta ,  \xi \rangle_{ W^{-1/s,s}(  \Gamma),  W^{1-1/s',s'}(  \Gamma)} , \quad \hbox{ for any }  \xi \in W^{1,s}_{\Gamma_D}(\Omega). $$

	%
	%
	%
	%
	%

	\medskip\noindent 	
	Our main result here is the following

	\begin{theorem}\label{Tdualhomogeneous}
		For any $0\leq \tilde\rho\in L^s(\Omega)$ , we have  
		\begin{equation}
			\N (\tilde\rho)   = \max_{ p\in \G_\kk }\Big\{ \int_\Omega  \tilde \rho \: p\: \dd x + \tau\: 	 \int_{\Gamma_N} \eta\: p\: \dd x -\int_\Omega p^+\: \dd x   \Big\} :=D_g^\infty(\tilde \rho),
		\end{equation}
		where 
		$$\G_{\kk}:=\Big\{ z\in W^{1,\infty}(\Omega)\: :\:  z=g\hbox{ on }\Gamma_D\hbox{ and } \vert \nabla z(x) \vert \leq \kk(x) \hbox{ a.e. }x\in \Omega \Big\} .$$
		Moreover, 
		\begin{equation} \label{mininf}
			\N (\tilde\rho) = \min_{\rho\in \K_1}   \inf_{\tau \Phi\in \F (\tilde \rho-\rho)}\left \{ \tau  \int_\Omega \kk(x)\: \vert \Phi(x)\vert\: \dd x -\tau \int_{\Gamma_D} g(x)\: \Phi\cdot \bnu \:     \dd x   \right\} ,
		\end{equation} 
		and,   if $\rho$ and $p$ are optimal solutions     of  both problems 	$\N (\tilde\rho) $  and  $D_g^\infty(\tilde \rho)$ respectively, then  	$p\in \G_\kk $, $\rho\in \K_1,$     $\rho\in \hbox{Sign}^+(p),$  a.e. in $\Omega $    and 		 
		\begin{equation} 
			\int_\Omega (\tilde \rho -\rho)\: (p-\xi)\:   \dd x -\tau\: \int_{\Gamma_N} \eta\: (p-\xi)\:  \dd x  \geq 0 ,\quad \hbox{ for any }\xi\in \G_\kk .	\end{equation} 
	\end{theorem} 
	
	\begin{remark}
		Thanks to Theorem \ref{Tdualhomogeneous}, one sees that  the condition \eqref{condg} is a sufficient and necessary condition. 	 It is more or less well known by now that this condition is equivalent to the fact that (see for instance \cite{EINHJ})  
		$$g(x)-(y)\leq \min\left\{ \int_0^1 \kk(\varphi(t))\: \vert \dot {\varphi} (t)\vert \: \dd t \: :\:  \varphi \in \Lip([0,1],\Omega),\: \varphi(0)=x,\: \varphi(1)=y\right\}.  $$	 
	\end{remark}

	\medskip
	To prove Theorem \ref{Tdualhomogeneous},  we use  Von Neumann-Fan minimax  theorem that we remind below 
	
	\medskip
	\begin{Rtheorem} [Von Neumann-Fan minimax theorem, see for instance \cite{Borwein&Zhuang}]
		Let $X$ and $Y$  be Banach spaces. 	Let $C\subset X$ be nonempty and convex, and let $D\subset Y$ be nonempty, weakly compact 	and convex. Let  $g\: :\:  X\times Y\to \RR$ be convex with respect to  $x\in C$ and concave and upper-semicontinuous with respect to $y\in D,$  and weakly continuous in $y$ when
		restricted to  $D.$ Then
		$$\max_{y\in D}\inf_{x\in C} g(x,y)=  \inf_{x\in C}  \max_{y\in D} g(x,y).  $$
	\end{Rtheorem}

	\medskip
	To this aim, we use the following result which goes back to \cite{EINHJ} in the case where $\eta\equiv 0.$  For completeness, a proof  is given in  Appendix-\ref{subsection:duality}.
	
	\begin{lemma}
		\label{phj}
		We have  
		\begin{equation}\label{std1}
			\begin{array}{l}
				\inf_{\tau \Phi\in \F (\tilde \rho-\rho)}\left \{ \tau  \int_\Omega \kk(x)\: \vert \Phi(x)\vert\: \dd x -\tau  \int_{\Gamma_D} g(x)\: \Phi\cdot \bnu \:     \dd x   \right\} \\  \\
				\hspace*{1cm}= 	\max_{p\in \G_\kk}\left \{   \int (\tilde \rho-\rho)\: p\:  \dd x   - \tau\:  \int_{\Gamma_N} p\: \eta \: \dd x \right\}   . 
			\end{array} 
		\end{equation} 
	\end{lemma}

	\bigskip
	\begin{proof}[Proof of Theorem \ref{Tdualhomogeneous}]   Since $\tilde \rho-\rho\in L^s(\Omega)$ and $\eta \in W^{-\frac{1}{s},s}(\Gamma_N)$, we know that    $\F (\tilde \rho-\rho)\neq\emptyset$.  Moreover, since 
		\begin{equation}\label{simple1}
			\N (\tilde\rho) = \inf_{\rho\in \K_1}  \inf_{\tau \Phi\in \F (\tilde \rho-\rho)}\left \{  \tau \int_\Omega \kk(x)\: \vert \Phi(x)\vert\: \dd x -\tau \int_{\Gamma_D} g(x)\: \Phi\cdot \bnu \:     \dd x   \right\},
		\end{equation}   
		by   Lemma \ref{phj},  we get 
		\begin{equation}\label{fromHJ}
			\N (\tilde\rho)   =    \inf_{\rho\in \K_1} 
			\max_{p\in \G_\kk}\left \{   \int (\tilde \rho-\rho)\: p\:  \dd x -\tau \:  \int_{\Gamma_N} p\: \eta \: \dd x \right\}.
		\end{equation}
		Using	 Von Neumann-Fan minimax theorem as in \cite{Borwein&Zhuang}, we deduce that    
		\begin{eqnarray}
			\N (\tilde\rho)  &=& \max_{p\in \G_\kk}\inf_{\rho \in \K_1}   \left \{   \int (\tilde \rho-\rho)\: p\:  \dd x -\tau\:  \int_{\Gamma_N} p\: \eta \: \dd x \right\}  \\ \\
			&=& \max_{p\in \G_\kk}  \left \{ \int  \tilde \rho \: p\:  \dd x - \int  p^+\:  \dd x  -\tau\:  \int_{\Gamma_N} p\: \eta \: \dd x  \right\}   . 
		\end{eqnarray} 
		Taking 
		$$p=  \argmax_{p\in W^{1,\infty}(\Omega) ,\: p_{\mid\Gamma_D}=g}  \left \{ \int  \tilde \rho \: p\:  \dd x - \int  p^+\:  \dd x  + \tau\: \int_{\Gamma_N} p\: \eta \: \dd x   \right\} 
		$$ and 
		$\rho=\hbox{Sign}_0^+(p)$  and using Lemma \ref{phj}, we deduce the equivalence between the solutions  of \eqref{std1} and \eqref{system2}.   	Thus the result of the theorem.

	\end{proof}

	\begin{remark}
		\begin{enumerate}
			\item It is known that the optimal flux  in \eqref{mininf} is not reached for a Lebesgue vector valued function $\Phi.$ Indeed, since the structure  of $F,$ one expects the optimal flux to be a  Radon measure  vector valued function $\Phi$. However, if  this is true and  if $(\rho,\Phi)$  and $p$     are solutions of  both problems $\N (\tilde\rho) $ and  $D(\tilde \rho)$ respectively, then     $\rho\in \hbox{Sign}^+(p),$  a.e. in $\Omega,$   $\Phi\cdot \nabla p =\kk\:  \vert \Phi\vert $ in $\Omega$   and 		 
			\begin{equation} 
				\tau\: \int_\Omega \Phi\cdot \nabla \xi\: \dd x = \int_\Omega (\tilde \rho -\rho)\: \xi~\dd x + \tau\: \int_{\Gamma_N} \eta\: \xi\: \dd x,\quad \hbox{ for any }\xi\in W^{1,s'}_{\Gamma_D}(\Omega).	\end{equation} 
			In general, one needs to be careful with the treatment of $\Phi\cdot \nabla p,$ since $\Phi$ is not   regular in general.  Here one needs, to use the notion of tangential gradient of $p$ (see \eg \cite{Bouchitte&al}) to handle the related PDE.

			\item  In connection with Evans-Gangbo formulation, the corresponding PDE may be written   as 
			\begin{equation} \label{statpde2}
				\left\{
				\begin{array}{ll}
					\displaystyle \tilde \rho -\tau\:  \dive(  \rho \:  W  )=  0,\quad  W = m \:  \nabla p\\  \\
					m\geq 0,\: p\geq 0,\: 0\leq \rho\leq 1,\: p(\rho-1)=0\\  \\
					\vert \nabla p\vert \leq \kk,\: m(\vert \nabla p\vert -\kk)=0
				\end{array}
				\right\} \quad \hbox{ in }  \Omega, 
			\end{equation} 
			subject to boundary condition 
			\begin{equation} 
				\left\{ 
				\begin{array}{ll}
					\displaystyle  \Phi\cdot \bnu =\eta \quad & \hbox{ on }(0,\infty)\times \Gamma_N, \\  \\
					p=g& \hbox{ on }(0,\infty)\times \Gamma_D.
				\end{array}
				\right.  
			\end{equation}

			\item 	As a formal consequence of Theorem \ref{Tdualhomogeneous}, under  the assumptions (H1)-(H2), the algorithm in Section \ref{Salgorithm} turns out  in solving successively two PDEs, a transport equation and a nonlinear second order equation.   This enables also to establish a continuous  model in terms of nonlinear PDE. This is summarized  in the following items.  
			\begin{enumerate} 
				\item The sequence $\rho_{1/2},\rho_{1},\dotsc,\rho_{k},\rho_{k+1/2},\rho_{k+1},\dotsc,\rho_{n}$ given by the algorithm in Section \ref{Salgorithm} is characterized by:  for each  $k=0,\dotsc,n-1,$ we have

				\begin{itemize}
					\item   \textbf{Prediction:} $ \rho_{k+\frac{1}{2}} =\rho(t_{k+\frac{1}{2}}) ,$ 
					where $\rho$ is the solution of the transport equation  : 
					\begin{equation} \partial_t \rho+\dive(\rho\: V )=0\quad \hbox{ in }[t_k,t_{k+\frac{1}{2}}[,
					\end{equation}  with  $ \rho(t_k)=\rho_k,$   	$V$ is a  given vector field. For instance $V=-\nabla \varphi/\Vert \nabla \varphi \Vert$ and $\varphi$ is  the solution of   the eikonal equation \eqref{hjk}. 
					
					\item  \textbf{Correction:}  $\rho_{k+1}$ is a   solution of the PDE

					\begin{equation} 	\left\{
						\begin{array}{ll}
							\left. 	\begin{array}{l}
								\rho_{k+1}  - \tau\: \dive( \rho_{k+1}\:  W) =\rho_{k+1/2} ,  \quad 
								W = m \:  \nabla p_{k+1} \\  \\
								m\geq 0,\: p_{k+1}\geq 0,\: 0\leq 	\rho_{k+1}  \leq 1,\: p_{k+1}(\rho_{k+1}-1)=0\\  \\
								\vert \nabla p_{k+1}\vert \leq \kk,\:  m(\kk-\vert \nabla p_{k+1}\vert)=0   
							\end{array} \right\}  \quad & \hbox{ in } \Omega,\\  \\
							\Phi \cdot \bnu =\eta  & \hbox{ on }\Gamma_N,\\  \\
							p_{k+1}=g    & \hbox{ on }\Gamma_D.
						\end{array}
						\right.
					\end{equation} 
					
				\end{itemize}		 
				
				\item  Considering the application $\rho_\tau \: :\: [0,T)\to L^\infty(\Omega)$ and $p_\tau \: :\: [0,T)\to W^{1,\infty}(\Omega)$  given by 
				$$\rho_\tau(t)=\left\{  \begin{array}{ll}
					\rho_{k+\frac{1}{2}}   \quad &\hbox{ for any } t\in [t_k,t_{k+\frac{1}{2}}[ \\ \\  
					\rho_{k+1 } &\hbox{ for any } t\in [t_{k+\frac{1}{2}},t_{k+1}[ 
				\end{array}\right.   \quad \hbox{ for } k=0,1,\dotsc, n-1,  $$
				and 	$$p_\tau(t)=\left\{  \begin{array}{ll}
					0  \quad &\hbox{ for any } t\in [t_k,t_{k+\frac{1}{2}}[ \\ \\  
					p_{k+1 } &\hbox{ for any } t\in [t_{k+\frac{1}{2}},t_{k+1}[ 
				\end{array}\right.   \quad \hbox{ for } k=0,1,\dotsc, n-1,  $$ one expects that 
				\begin{itemize}
					\item $\rho_\tau   \to  \rho $  and $p_\tau   \to  p $ as $\tau\to 0,$  
					
					\item the couple $(\rho,p)$ satisfies the following evolution PDE 
					
					\begin{equation} \label{evolgran}
						\left\{
						\begin{array}{ll}
							\displaystyle \frac{\partial \rho }{\partial t}  +\dive(  \rho  \: (V- W)   )=  0,\quad  W =m  \:  \nabla p \\  \\
							m\geq 0,\: p\geq 0,\: 0\leq \rho\leq 1,\: p(\rho-1)=0\\  \\
							\vert \nabla p\vert \leq\kk,\: m(\vert \nabla p\vert -\kk)=0
						\end{array}
						\right\} \quad \hbox{ in } (0,\infty)\times \Omega. 
					\end{equation} 
					subject to boundary condition 
					\begin{equation} 
						\left\{ 
						\begin{array}{ll}
							\displaystyle  \Phi\cdot \bnu =\eta \quad & \hbox{ on }(0,\infty)\times \Gamma_N \\  \\
							p=g& \hbox{ on }(0,\infty)\times \Gamma_D.
						\end{array}
						\right.  
					\end{equation}   
					
				\end{itemize}

			\end{enumerate}		  
			
			\item See that the patch $W$ is null  outside  the congestion zone $[\rho=1].$

			\item  	Remember here, that the main operator which governs the correction step in this case,   given by 
			\begin{equation} 
				\left\{
				\begin{array}{ll}
					-\nabla \cdot (m  \:  \nabla p)  = \mu  \\  \\
					m\geq 0,\:  
					\vert \nabla p\vert \leq \kk,\: m(\vert \nabla p\vert -\kk)=0,
				\end{array}
				\right.  
			\end{equation} 
			is well known in the study of sandpile  (see \cite{IgDu} and the references therein). The dynamic here is connected to  a granular one. In other words the individuals behaves like grains of sand (see \cite{EvReza} and also \cite{Igstoch} for a stochastic microscopic description of the granular dynamic),  in the congestion zone and not like a fluid as follows from the quadratic case.

		\end{enumerate}
		
	\end{remark}

	\begin{remark}
		After all, the nonlinear PDE \eqref{evolgran} is a new model   we propose for the description of dynamical population where the movement of the agent is of   granular type like in sandpile.    In this paper, 
		we are proposing  its numerical study.   The theoretical study is a challenging problem for existence and uniqueness. This is an open problem and will not be treated in this paper.  Recall that, the case where the PDE is of diffusive type the PDE is well used and studied. There is a huge literature on this case,  one can see the recent paper \cite{Igshuniq} and the references therein for more details.   
		
	\end{remark}

	\begin{remark}
		In the  case where $V$ is computed just before the $k-$th prediction step by taking the speedy path given  the following eikonal equation 
		\begin{equation}\label{hjp}
			\left\{ \begin{array}{ll}
				\Vert \nabla \varphi\Vert =H(p_k)\quad & \hbox{ in }\Omega,\\   \\
				\varphi =0 & \hbox{ on }\Gamma_D,
			\end{array}
			\right.
		\end{equation}
		where $H$ is a given positive continuous function,  the  evolution problem \eqref{evolgran} needs to be coupled with the system 
		\begin{equation} 
			\left\{ \begin{array}{ll}
				V=-\nabla \varphi/\Vert \nabla \varphi \Vert \quad & \hbox{ in }\Omega \\  	\Vert \nabla \varphi\Vert =H(p)\quad & \hbox{ in }\Omega,\\   \\
				\varphi =0 & \hbox{ on }\Gamma_D.
			\end{array}
			\right.
		\end{equation}
		This is an interesting variant of Hugues model where the speedy path is computed by taking into account the congestion of the crowd. Indeed, taking $H$ a continuous function such that $H(p)$ is take instantaneously large value for positive $p,$ enables to avoid congestion zones.   From theoretical point of view, the eikonal equation turns out to be a well posed and stable problem since $p$ and then $H(p)$ are    regular, rather than $\rho$ as in Hugues model.     To improve the algorithm, we take in some numerical computation $f=f(p)$ in \eqref{hjk}  to compute   the spontaneous velocity field $V$.  The theoretical study of the corresponding evolution problem will be treated in forthcoming works. 
	\end{remark}

	\begin{remark} [Quadratic case]\label{remark:quadraric_case}
		Before to end up this section let summarize here some  formal results  concerning the quadratic case. The proofs may be found  in   \cite{IgCrowd} where the second author study some connected dynamic in the case of two populations.   	The quadratic case corresponds to  
		\[
		F(x,\xi) =\frac{1}{2} \vert \xi\vert^2, \hbox{ for any }(x,\xi)\in \Omega \times \RR^N.
		\] 
			The infimum in \eqref{argmink}   coincides with 
			\begin{equation}\label{mindual2}
				\inf_{p}\left\{ \int p^+(x) \dd x+ \frac{1}{2} \int   \vert  \nabla p(x) \vert^{2} \: \dd x - \int  p(x)\:  \tilde \rho(x)  \: \dd x:~p\in H^1(\Omega) ,\ p=g\hbox{ on }\Gamma_D   \right\} .
			\end{equation}
			Moreover, $p$ and $(\rho,\Phi)$  are solutions of  both problems  respectively, if and only if $(p,\rho,\Phi)$ is a solution of the following PDE
			\begin{equation}
				\label{system3}
				\left\{
				\begin{array}{ll}
					\left. 	\begin{array}{l}
						\rho - \dive( \Phi) = \tilde \rho,\quad   \Phi =  \nabla p    \\   \\
						\rho \in \hbox{Sign}^+(p)
					\end{array} \right\}  \quad & \hbox{ in } \Omega,\\  \\
					\Phi \cdot \bnu =\eta & \hbox{ on }\Gamma_N,\\  \\
					p=g    & \hbox{ on }\Gamma_D.
				\end{array}
				\right.
			\end{equation}
			In some sense, this implies that   the quadratic case is closely connected to  the system \eqref{Maury1} which was proposed by  B.  Maury \etal  (\cf  \cite{MRS1}) in the framework of gradient flows in the Wasserstein space of probability measures.  And, moreover, 	  the correction step corresponds simply to the time Euler-Implicit discretization for the diffusion process in 
			\eqref{Maury1}.

			%
			
			%
		
	\end{remark}
	
	\begin{remark} 
		\begin{enumerate}
			\item 
			Notice here that even though our approaches (based on  minimum flow problem),  provide the same continuous dynamics (at least in the quadratic case) with  gradient flow in the Wasserstein space of probability measures, both approaches are not the same at discrete level. While, the correction  with this approach   is recovered by a projection  with respect to $\Wass_2$ on the set $\{\rho\in L^\infty(\Omega)\: :\: 0\leq \rho\leq 1\hbox{ a.e. in }\Omega  \},$   our approach provides the correction by solving    an  elliptic problem through a minimum flow problem.   As far as we know, these are not the same even though one can be considered as an approximation of the other. 
			
			
			\item 	 
			In contrast of the quadratic case, in the homogeneous case  we do believe here that  recovering 	the correction by a projection  with respect to $\Wass_1$ on the set $\{\rho\in L^\infty(\Omega)\: :\: 0\leq \rho\leq 1\hbox{ a.e. in }\Omega  \},$ or by using our approach 
			are the same in the homogeneous case. 	This issue would be treated in forthcoming works.

		\end{enumerate}
		
	\end{remark}

	\section{Numerical approximation}\label{section:numerics}
	
	\subsection{Formulation and discretization} As discussed in Section-\ref{section:model}, the approximation of the density $\rho$ is performed via a prediction-correction strategy. The first step (prediction) consists in the resolution of the continuity equation \eqref{eq:continuity} which will be done using an Euler scheme for the time discretization, whereas the term $\dive(V\rho)$ is discretized using finite volumes. The second step (correction or projection) relies on  a minimum flow  problem which will be solved using a primal dual algorithm (PD). To begin with, let us give details concerning the discretization of the problems \eqref{eq:continuity}-\eqref{argmink}.
	
	\textbf{Domain discretization:} 
	In this section, we solve numerically $\eqref{eq:continuity} $ and \eqref{argmink} on the domain $\Omega$ shown on Figure \ref{fig:domain}. This domain represents a room surrounded by walls which we call $\Gamma_N$ and has an exit door $\Gamma_D$. The domain is divided into a set of $m \times n$ control volumes of length $h$ and width equal to $h$. We denote by $C_{i,j}$ the cell at the position $(i,j)$ and by $\Psi_{i,j}$ is the average value of the quantity $\Psi$ on $C_{i,j}$. At the interface of $C_{i,j}$, $\omega_{i+\frac{1}{2},j}$, $\omega_{i-\frac{1}{2},j}$, $\omega_{i,j+\frac{1}{2}}$ and  $\omega_{i,j-\frac{1}{2}}$ are the in/out flow quantities (see Figure-\ref{fig:domain}).\\ 
	\begin{figure}[H]
		\includegraphics[scale=0.27]{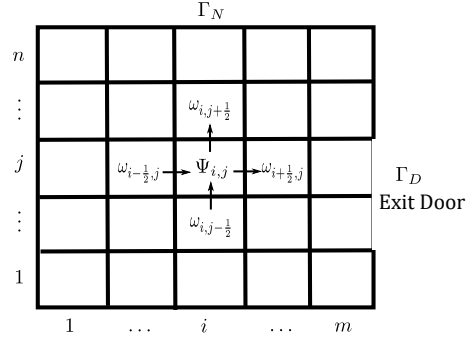}
		\caption{Discretization of the domain $\Omega$. }
		\label{fig:domain}
	\end{figure}
	
	We define discrete divergence is defined by:
	\begin{equation}\label{eq:div1}
		(\odiv_h\Phi)_{i, j} = \frac{\Phi^{1}_{i+\frac{1}{2},j} - \Phi^{1}_{i-\frac{1}{2},j}}{h}  +\frac{\Phi^{2}_{i,j+\frac{1}{2}} - \Phi^{2}_{i,j-\frac{1}{2}}}{h}.
	\end{equation}
	To take into account the Neumann boundary condition $\Phi\cdot\bnu = 0$ on $\Gamma_N$, we  impose:
	\begin{itemize}
		\item $\Phi_{\frac{1}{2},j}^1=0$, for $1\leq j \leq n$, 
		\item $\Phi_{m+\frac{1}{2},j}^1=0 ,  ~ \text{if} \: ((m+\frac{1}{2} )h,j h)  \in \Gamma_N$ ,    
		\item $\Phi_{i,\frac{1}{2}}^2=0$, for $1\leq i \leq m$,
		\item $\Phi_{i,n+\frac{1}{2}}^2=0$, for $1\leq i \leq m.$
	\end{itemize}
	We can rewrite this in a more compact way
	\begin{equation}
		\label{eq:div_matrix}
		\begin{aligned}
			(\odiv_h\Phi)_{i,j}^1&=D^{1}_{p}\Phi^{1}_{i,j}, ~\text{if} \: ((m+\frac{1}{2} )h,j h)  \in \Gamma_D,\\
			(\odiv_h\Phi)_{i,j}^1&=D^1_m\Phi^{1}_{i,j}, ~\text{if} \: ((m+\frac{1}{2} )h,j h)  \in \Gamma_N,\\
			(\odiv_h\Phi)_{i,j}^2&=D^2\Phi^{2}_{i,j},
		\end{aligned}
	\end{equation}
	where the matrices $D^{1}_{m},D^{1}_{p}, D^{2}$ are recalled in Appendix-\ref{subsection:discrete_op}.
	Then,  we define the discrete gradient operator as follows:
	
	\begin{equation}
		\label{eq:grad_h}
		\begin{aligned}
			(\nabla_{h}p)^1_{i,j}&= ~-^tD^1_p p(i,j) ,~\text{if} \: ((m+\frac{1}{2} )h,j h)  \in \Gamma_D,\\
			(\nabla_{h}p)^1_{i,j}&= ~-^tD^1_m p(i,j) ,  ~\text{if} \: ((m+\frac{1}{2} )h,j h)  \in \Gamma_N,\\
			(\nabla_{h}p)^2_{i,j}&= ~-^tD^2 p(i,j).
		\end{aligned}
	\end{equation}
	This being said, one can easily check that $\dive_h = -\nabla_h^{*}$.\\

	\textbf{Discretization of the transport equation \eqref{eq:continuity} :}
	We use a splitting method as follows. Given a final time $T>0$ and a timestep $\tau >0$, we decompose the interval $[0,T]$ into subintervals $[t_k,t_{k+\frac{1}{2}}]$ and $[t_{k+\frac{1}{2}}, t_{k+1}]$, with $k=0,\dotsc,n-1$. On each interval $[t_k,t_{k+\frac{1}{2}}]$ we solve the following continuity equation
	\begin{equation}
		\label{eq:continuity1}
		\left\lbrace
		\begin{aligned}
			\partial_t \rho + \dive(V\rho)&= 0\\
			\rho(t_k) &=\rho^{k-1},
		\end{aligned}
		\right.
	\end{equation}
	to obtain $\rho^{k+\frac{1}{2}}$, where $V=(V^x,V^y)$ is the velocity field given by $V = -\nabla\mathbf{D}/\Vert \nabla\mathbf{D}\Vert$,  and $\mathbf{D}$ being the distance (not necessary euclidean) to the boundary $\Gamma_D$  given by the eikonal equation \eqref{hjk} whose resolution is recalled in Appendix-\ref{subsection:eikonal}. Solving \eqref{eq:continuity1} can be done by combining a finite difference method in the time variable combined with a $2$D finite volume method in the space variable. We approximate the term $\dive(V\rho)$ in the cell $C_{i,j}=[x_{i-\frac{1}{2},j},x_{i+\frac{1}{2},j}]\times [y_{i,j-\frac{1}{2}},y_{i,j+\frac{1}{2}}]$ as follow:
	\begin{equation*}
		(\dive(V\rho))_{i,j}=\frac{1}{\Delta x} [\rho_{i+\frac{1}{2},j} V^{x}_{i+\frac{1}{2},j}-\rho_{i-\frac{1}{2},j} V^{x}_{i,j}]+\frac{1}{\Delta y} [\rho_{i,j+\frac{1}{2}} V^{y}_{i,j+\frac{1}{2}}-\rho_{i,j-\frac{1}{2}} V^{x}_{i,j-\frac{1}{2}}],
	\end{equation*}  
	where $(\dive(V\rho))_{i,j}$ the value of $\dive(V\rho)$ in the cell $C_{i,j}$ and $(\Delta x,\Delta y)$ are the spatial discretization. Notice that in practice, we take $\Delta x = \Delta y = h$, where $h$ is the mesh size introduced above.
	
	For the time disctization, we use the Euler explicit method to approximate the time derivative of the density. The overall scheme can the be written as:
	
	\begin{equation}
		\label{eq:disc_continuity}
		\frac{\rho_{i,j}^{k+\frac{1}{2}}-\rho_{i,j}^{k}}{\tau}+\frac{1}{\Delta x} [\rho_{i+\frac{1}{2},j}^{k} V^{x}_{i+\frac{1}{2},j}-\rho_{i-\frac{1}{2},j} ^{k}V^{x}_{i-\frac{1}{2},j}]+\frac{1}{\Delta y} [\rho_{i,j+\frac{1}{2}}^{k} V^{y}_{i,j+\frac{1}{2}}-\rho_{i,j-\frac{1}{2}}^{k}V^{x}_{i,j-\frac{1}{2}}] = 0
	\end{equation}
	where $\rho_{i,j}^{k+\frac{1}{2}}$ is the average value of $\rho$ in the cell $C_{i,j}=[x_{i-\frac{1}{2},j},x_{i+\frac{1}{2},j}]\times [y_{i,j-\frac{1}{2}},y_{i,j+\frac{1}{2}}]$ at time $({k+\frac{1}{2}})\tau$, and $\rho_{i+\frac{1}{2},j}^{k}$, $V^x_{i+\frac{1}{2},j}$ are the values of $\rho$ and $V$at the interface $x_{i+\frac{1}{2},j}$ at time $k\tau$ respectively. Similarly, $(\rho_{i-\frac{1}{2},j} ^{k} V^{x}_{i-\frac{1}{2},j})$, $(\rho_{i,j+\frac{1}{2}}^{k}, V^{y}_{i,j+\frac{1}{2}}) $ and $(\rho_{i,j-\frac{1}{2}}^{k},V^{x}_{i,j-\frac{1}{2}})$ are the values, at time $\tau k$, of $(\rho,V)$ at the interface $x_{i-\frac{1}{2},j}$, $y_{i,j+\frac{1}{2}}$ and $y_{i,j-\frac{1}{2}}$ respectively.

	Using the upwind scheme we have  $\rho_{i+\frac{1}{2},j}^{k}=\rho_{i,j}^{k}$ and $\rho_{i,j+\frac{1}{2}}^{k}=\rho_{i,j}^{k}$.  Substituting in \eqref{eq:disc_continuity},  the density $\rho_{i,j}^{k+\frac{1}{2}}$ can be written as:
	
	\begin{equation}
		\label{sol:continuity}
		\rho_{i,j}^{k+\frac{1}{2}}=\rho_{i,j}^{k}-\frac{\tau}{\Delta x} [\rho_{i,j}^{k} V^{x}_{i+\frac{1}{2},j}-\rho_{i-1,j} ^{k}V^{x}_{i-\frac{1}{2},j}]-\frac{\tau}{\Delta y} [\rho_{i,j}^{k} V^{y}_{i,j+\frac{1}{2}}-\rho_{i,j-1}^{k}V^{x}_{i,j-\frac{1}{2}}] 
	\end{equation}

	We  consider that no flux is entering the room from the walls at $\Gamma_N$. This is equivalent to impose $\rho_{i-\frac{1}{2},j} ^{k}V^{x}_{i-\frac{1}{2},j}=0$ and $\rho_{i,j-\frac{1}{2}} ^{k}V^{y}_{i,j-\frac{1}{2}}=0$ at $i=1$ and $j=1$ respectively.\\
	Finally, let us recall that the values of $h$ and $\tau$ are chosen to satisfy a CFL-type constraint $\max(\Vert V_{i,j}\Vert)\frac{\tau}{h}<\frac{1}{2}$ in order to guarantee the stability of the numerical scheme \eqref{eq:disc_continuity}. We summarize this in the following algorithm:
	\begin{algorithm}[H]
		
		\caption{Prediction step}
		\label{alg:prediction}
		\begin{algorithmic}
			\State $\textbf{1st step.}$ Initialization:  Compute the velocity  $V=(V^x,V^y)$. Choose $\Delta x=\Delta y=h$ and  $\tau$ such $\max(\Vert V_{i,j}\Vert)\frac{\tau}{h}<\frac{1}{2}$  and take a initial density given by $\rho_{i,j}^{k}$ at time $k \tau$.
			\State $\textbf{2nd step.}$ Update the density at time $(k+\frac{1}{2}) \tau$ by
			$$
			\begin{aligned}
				\rho_{i,j}^{k+\frac{1}{2}}=\rho_{i,j}^{k}-\frac{\tau}{\Delta x} [\rho_{i,j}^{k} V^{x}_{i+\frac{1}{2},j}-\rho_{i-1,j} ^{k}V^{x}_{i-\frac{1}{2},j}]-\frac{\tau}{\Delta y} [\rho_{i,j}^{k} V^{y}_{i,j+\frac{1}{2}}-\rho_{i,j-1}^{k}V^{x}_{i,j-\frac{1}{2}}] .
			\end{aligned}
			$$
		\end{algorithmic}
	\end{algorithm}
	\begin{remark}
		The discretization of $\dive(V \: \rho)$ assumes a positive direction for the speed \ie $V^x>0$ and $V^y>0$. However, the scheme can be easily adapted to other cases. For example, if $V^x>0$ and $V^y<0$ for some $(i,j)$, the discretization of $\dive(\rho V)$ becomes:
		\begin{equation}
			\label{divergence2}
			(\dive(V\rho))_{i,j}=\frac{1}{\Delta x} [\rho_{i+\frac{1}{2},j} V^{x}_{i+\frac{1}{2},j}-\rho_{i-\frac{1}{2},j} V^{x}_{i-\frac{1}{2},j}]+\frac{1}{\Delta y} [\rho_{i,j-\frac{1}{2}} V^{y}_{i,j-\frac{1}{2}}-\rho_{i,j+\frac{1}{2}} V^{x}_{i,j+\frac{1}{2}}],
		\end{equation}
	\end{remark}
	\color{black}

	\bigskip 
	Since the obtained density $\rho^{k+\frac{1}{2}}$  may violate the constraint $\rho \leq1$, the next step is to handle congestion by solving the following minimum flow problem
	\begin{equation}
		\label{beckmann1}
		\inf_{(\rho,\Phi)}\left\{  \int_\Om \kk(x)\vert\Phi(x)\vert\dd x:~-\tau \:\dive(\Phi) = \rho^{k+\frac{1}{2}} - \rho\hbox{ in }  \Omega,\ \Phi\cdot \nu=0\hbox{ on }\Gamma_N \mbox{ and } 0<\rho\leq 1\right\},
	\end{equation}
	where  where $\kk\geq 0$ is a continuous function and, for the simplicity of the presentation, we take vanishing $g$ and $\eta$ (see Remark \ref{Nonhometa}).\\

	\textbf{Discretization of the minimum flow problem \eqref{beckmann1} : }First, let us rewrite   \eqref{beckmann1}   in the form
	\begin{equation}
		\label{formulation}
		\M:~~\min_{(\rho,\Phi)} \A(\rho,\Phi) +  \I_{\CC} (\Lambda(\rho,\Phi)),
	\end{equation}
	where (we omit the variable $\tau$ to lighten the notation)
	\[
	\A(\rho,\Phi) =\int_\Om \tau\kk(x)\vert\Phi(x)\vert\dd x + \I_{[0,1]}(\rho),\quad  \Lambda(\rho,\Phi)  = \rho-\tau\dive\Phi  \quad \mbox{and}~\B  = \I_{\left \{  \rho^{k+\frac{1}{2}}\right\}} .
	\]
	This  problem    can be efficiently solved by Chambolle-Pock's primal-dual algorithm (PD) (\cf \cite{c2004}). 
	
	Based on the discrete gradient and divergence operators, we propose a discrete version of $\M$ as follows
	\begin{equation}
		\M_{d}:~\min\limits_{\substack{(\rho,\Phi)}}\Big\{h^2\sum_{i=1}^{m+1}\sum_{j=1}^{n+1}   \tau\kk_{i,j}\Vert \Phi_{i,j}\Vert +\I_{[0,1]}(\rho)+ \I_{\CC}(\Lambda_h (\rho,\Phi)) \Big\} 
	\end{equation}
	where $\CC:=\left\{ (a_{i,j})\: : \: a_{i,j} = \rho_{i,j}^{k+\frac{1}{2}},~~ \forall (i,j)\in\llbracket1,m\rrbracket\times\llbracket1,n\rrbracket \right\},$  $\Lambda_h (\rho,\Phi) = \rho -\tau\dive_h\Phi $ and $\kk_{i,j}$ is the value of $\kk$ in $C_{i,j}$.	In other words, the discrete version $\M_d$ can be written as
	\begin{equation}
		\label{P1}
		\min_{(\rho,\Phi)} \A_h(\rho,\Phi) + \B_h(\Lambda_h(\rho,\Phi)),
	\end{equation}
	or in a primal-dual form as
	\begin{equation}
		\label{eq:P1PD}
		\min_{(\rho,\Phi)}\max_{p} \A_h(\rho,\Phi) + \langle u,\Lambda_{h}(\rho,\Phi)\rangle - \B_{h}^{*}(p),
	\end{equation}
	where
	\begin{equation}
		\label{eq:functionals_h}
		\A_h(\rho,\Phi) = h^2\sum_{i=1}^{m+1}\sum_{j=1}^{n+1}\tau\kk_{i,j}\Vert\Phi_{i,j}\Vert  + \I_{[0,1]}(\rho)~ \mbox{and}~\B_h=\mathbb{I}_{\CC}.
	\end{equation}
	
	Notice that in this case, \eqref{P1} has a dual problem that reads
	\begin{equation}
		\min\limits_{\substack{\rho\in X\\ 0\leq\rho\leq 1}}\max\limits_{\substack{p\in X\\p = 0 \text{ on }  \Gamma_D}} h^{2}\left\{\sum_{i=1}^{m}\sum_{j=1}^{n}p_{i,j}(\rho_{i,j}^{k+\frac{1}{2}} - \rho_{i,j}):~\Vert \nabla_{h} p_{i,j}\Vert \leq \kk_{i,j}\right\}.
	\end{equation}
	Then (PD) algorithm \cite{CP} can be applied to $\M_d$ as follows:
	\begin{algorithm}[H]
		
		\caption{(PD) iterations}
		\label{alg:pd1}
		\begin{algorithmic}
			\State $\textbf{1st step.}$ Initialization:  choose $\alpha,\beta>0$, $\theta\in [0,1]$, $\rho_0, \Phi^0$ and take $u_0 = \Lambda_h (\rho^0,\Phi^0), \: \bar{p}^0 = p^0$
			\State $\textbf{2nd step.}$ For $l\leq \mathrm{Iter}_{max}$ do
			$$
			\begin{aligned}
				(\rho^{l+1},\Phi^{l+1}) &= \prox_{\beta\A_h}\left((\rho^l,\Phi^l) -\beta\Lambda_{h}^{*}(\bar{p}^l)\right);\\
				p^{l+1} &= \prox_{\alpha\B_{h}^{*}}\left(p^l+\alpha \Lambda_{h}(\rho^{l+1},\Phi^{l+1})\right);\\
				\bar{p}^{l+1} &= p^{l+1} + \theta  (p^{l+1}- p^{l}).\\
			\end{aligned}
			$$
		\end{algorithmic}
	\end{algorithm}
	Recall here that the proximal operator   is defined through
	\begin{equation}
		\label{def:prox}
		\prox_{\alpha E}(p) =\underset{q}{\operatorname{argmin}} \frac{1}{2}\|p-q\|^{2}+\eta E(q).
	\end{equation}

	\subsection{Computation of the proximal operators}
	See that  for  the functional $\A_h$ and $\B_{h}^{*}$  can be computed explicitly. Indeed, the functional $\A_h$ is separable in the variables $\rho$ and $\Phi$ : 
	$$\A_h(\rho,\Phi)  = \I_{[0,1]}(\rho) + \Vert\Phi\Vert_{1}.$$ So, $\prox_{\eta \A_h}$ is the some of a projection in the first component and the so-called soft-thresholding. Namely
	\begin{equation}\label{eq:prox_F}
		\left(\prox_{\A_h}(\rho,\Phi)\right)_{i,j} = \left(\max(0,\min(1,\rho_{i,j})), \max(0,1 - \frac{1}{\vert\Phi_{i, j}\vert})\Phi_{i, j} \right) .
	\end{equation}
	As to $\B_{h}^{*},$ in order to compute $\prox_{\alpha\B^{*}_{h}}$, we make use of Moreau's identity
	\begin{equation}
		p = \prox_{\alpha\B^*_h}(p) + \alpha\prox_{\alpha^{-1}\B_h}(p/\eta),
	\end{equation}
	and the fact that $\prox_{\alpha^{-1}\B_h}(a,b)$ is given simply by  the projection onto $\CC$. 	Consequently,
	$$
	\left(\prox_{\alpha\B^*_h}(p)\right)_{i,j} = \left( p_{i,j} - \alpha\proj_{\CC_{i,j}}(p_{i,j}/\alpha)\right) .
	$$
	
	Thus, the details of \hyperref[alg:pd]{Algorithm 3} to solve $\M_d$ are as follow : 
	\begin{algorithm}[H]
		
		\caption{(PD) iterations for $\M_d$}
		\label{alg:pd}
		\begin{algorithmic}
			\State $\textbf{Initialization:}$ Let $k=0$, choose $\alpha,\beta>0$ such that $\alpha\beta\Vert\Lambda_{h}\Vert^2<1$. Choose $\rho^0,\Phi^0$ and  $p^0  = \bar{p}^0 = p_0$.
			\State $\textbf{Primal step:}$ \label{Primal_step}
			\begin{equation}
				(\rho_{i,j}^{l+1},\Phi_{i, j}^{l+1})  = \left(\max\Big(0,\min(1,\rho_{i,j}^{l}-\beta \bar{p}^{l}_{i,j})\Big), \max\Big(0,1 - \frac{1}{\vert\Phi_{i, j}^{l}-\beta\nabla_h \bar{p}_{i,j}^{l}\vert}\Big)\Big(\Phi_{i, j}^{l}-\beta\nabla_h \bar{p}_{i,j}^{l}\Big) \right).
			\end{equation}
			\State $\textbf{Dual step:}$
			\begin{equation}
				\begin{aligned}
					v^{l+1} &= p^{l} + \alpha\rho^{l+1}-\alpha \dive_h(\Phi^{l+1}).\\
					p^{l+1}_{i,j} &= v^{l+1}_{i,j} -\alpha\proj_{\CC_{i,j}}(v_{i,j}^{l+1}/\alpha), \: 1\leq i\leq m, 1\leq j\leq n.
				\end{aligned}
			\end{equation}
			\State $\textbf{Extragradient:}$
			\[
			\bar{p}^{l+1} = 2p^{l+1} - p^{l}.
			\]
		\end{algorithmic}
	\end{algorithm}

	It was shown in \cite{CP} that when $\theta = 1$ and $\alpha\beta\Vert\Lambda_h\Vert^2 < 1$, the sequence $\{(\rho^l,\Phi^l)\}$ converges to an optimal solution of $\M_d$. So in practice, we choose $\alpha>0$ and we take $\beta = 1/(\eta K^2)$, where $K$ is an upper bound of  $\Vert\Lambda_h\Vert$. More precisely, $K = \sqrt{\Vert \nabla_h \Vert^2 + \Vert \id_X\Vert^2}\equiv \Vert \Lambda_h\Vert$. The algorithm was implemented in Matlab and all the numerical examples below were executed on a  2,6 GHz CPU running macOs High Sierra system.

	\begin{remark}[Non-homogeneous Neuman boundary condition : non null $\eta$]\label{Nonhometa}
		In is not difficult to see that in the case of  non null $\eta,$ one can handle this case by considering $\eta$ as a source term on the boundary on $\Gamma_N.$  To avoid numerical computation for the correction  
		we propose to handle the condition 
		$$ -\tau\: \dive(\Phi)  =\rho_{k+1/2} - \rho  \hbox{ in }   \D'(  \Omega )    \hbox { and }  \Phi\cdot \bnu = \eta  \hbox{ on }\Gamma_N.$$      
		as 
		$$ -\tau\: \dive(\Phi)  =\rho_{k+1/2} +\tau\: \eta - \rho  \hbox{ in }   \D'(  \overline \Omega \setminus \Gamma_D)   .$$     
		In other words, at each iteration  we take $\rho^{l} +\tau\: \eta $   instead of $\rho^{l} $ in the Algorithms \hyperref[alg:pd]{1}-\hyperref[alg:pd1]{2}.
		
	\end{remark}

	\section{Numerical simulations}\label{section:examples}

	In this section we present several examples to illustrate our approach \footnote{Demonstration videos are available at \url{https://github.com/enhamza/crowd-motion}}. We first examine the scenario of  evacuation of a population $\rho_0$ from a the domain $\Om\subset\R^2$ via an exit $\Gamma_D$ with  different velocities. In the last two examples we compare our approach to the one in \cite{MRS1,MRS2}, the configuration in the first one is similar to the previous  ones, \ie the crowed is initially located in a part of the room $\Om$ and try to escape through the doors, while in the second example the domain $\Om$ is constituted by two rooms connected by a  "bridge". In all these examples, the velocity field $V$ derives from a potential $\varphi$ that is considered as the distance function to the door $\Gamma_D$ and is computed by solving the eikonal equation
	\begin{equation}
		\left\lbrace
		\begin{aligned}
			\Vert\nabla \varphi\Vert &= f(\x)\\
			\varphi_{\mid\Gamma_D} &= 0,
		\end{aligned}
		\right.
	\end{equation}
	using  the primal-dual method proposed in \cite{EINHJ} (see also \cite{EINsfs}), where $f\geq 0$ is a continuous function that will be precised for each example. All the tests of  this section are performed with a mesh size $h= 0.01$ and a timestep $\tau = 0.004$.  Moreover, the corresponding velocities are displayed in red. 
	\vspace*{-0.5cm}
	\subsection{One room evacuation}
	In this first example (\cf Figure-\ref{fig:example1}), the initial density $\rho_0$ is given by $\rho_0(\x) = \mathbf{1}_{S_1}(\x)+\mathbf{1}_{S_2}(\x)$ with $S_1 = [0,\frac{1}{2}]\times [0,\frac{1}{3}]$ and  $S_2 =[0,\frac{1}{2}]\times[\frac{2}{3},1] $. The exit is given by $\Gamma_D = \{1\}\times [0.4,0.6]$ and $f\equiv 1$.
		\vspace*{-8mm}
	\begin{figure}[H]
		\includegraphics[width=.84\textwidth,center]{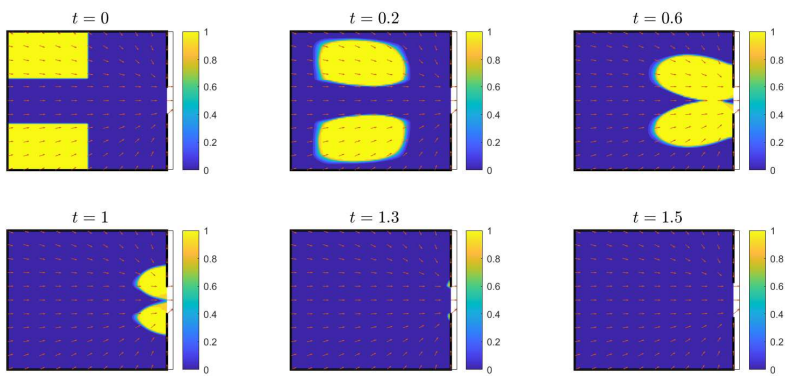}
		\caption{The crowed density $\rho$ computed at 6 different timesteps with  $T=2$ and $f\equiv 1$.}
		\label{fig:example1}
	\end{figure}
	
	In the second example (\cf Figure-\ref{fig:example2}),  the initial density is $\rho_0(\x) = \mathbf{1}_{S_1}(\x)$ with $S_1 = [0,\frac{1}{2}]\times [0,1]$ and $\Gamma_D = \left(\{1\}\times[0,0.4]\right)\cup\left(\{1\}\times[0.9,1]\right)$ and $f(\x) = e^{-3\times\left((x-\frac{1}{2})^2 + (y-\frac{1}{2})^2\right)}$.
	
	\begin{figure}[H]
		\includegraphics[width=.84\textwidth,center]{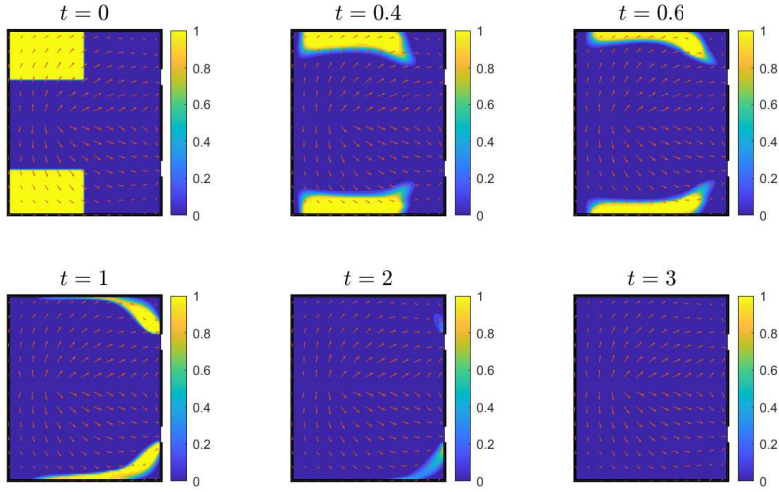}
		\caption{The crowed density $\rho$ computed at 6 different timesteps with  $T=3$ and $f(\x) = e^{-3\times\left((x-\frac{1}{2})^2 + (y-\frac{1}{2})^2\right)}$.}
		\label{fig:example2}
	\end{figure}
	In this example, the function $f$ has a bump in the middle of the domain, and we can observe in Figure-\ref{fig:example2} that the population is avoiding this region while heading the doors.\\

	\begin{figure}[!ht]
		\includegraphics[width=.84\textwidth,center]{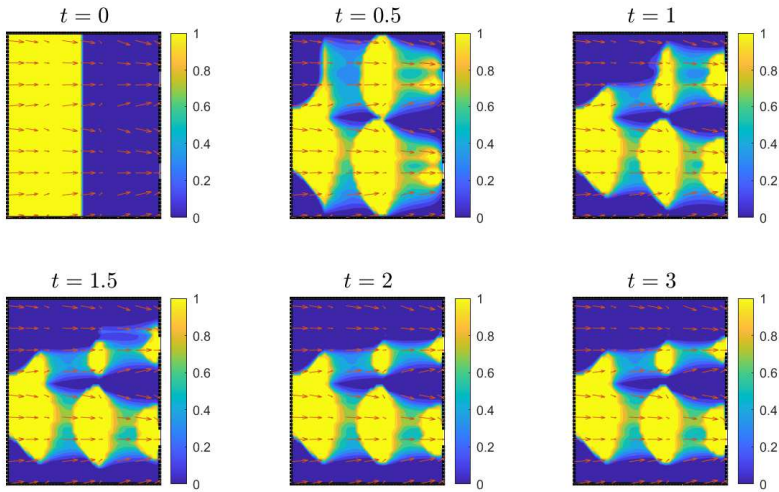}
		\caption{The crowed density $\rho$ computed at 6 different timesteps with  $T=3$ and  $f(\x) = |\cos(3x+5y)|+0.2$.}
		\label{fig:example3}
	\end{figure}
	In the third example (\cf Figure-\ref{fig:example3}),,   the initial condition for the density is $\rho_0(\x) = \mathbf{1}_{S_1}(\x)$ with $S_1 = [0,\frac{1}{2}]\times [0,1]$ and $\Gamma_D = \left(\{1\}\times[0.2,0.3]\right)\cup\left(\{1\}\times[0.7,0.8]\right)$ and $f(\x) = |\cos(3x+5y)|+0.2$. 	The source term is located on the entry of the domain at $\Gamma_S=\{0\}\times[0.3,0.6]$. \\ 
	 In this example one sees that the vector filed of spontaneous velocity has small values in   successive (periodic)  regions. This produce in turns   successive congestion zones. Moreover, the system reaches its equilibrium after $t=2$. One can notice that no variation in the density is observed as the number of persons leaving the room is equal to the number of person entering the room.

	\color{black}
	\subsection{Homogeneous case vs quadratic case}
	As we pointed out in Subsection-\ref{subsection:pde}, in the case where $F(x,\xi) =\vert \xi\vert$, our model is connected to the gradient flow in the Wasserstein space equipped with $\Wass_1$. Whereas the case $F(x,\xi) = \frac{1}{2}\vert\xi\vert^{2}$ can be related to the gradient flow in the Wasserstein space equipped with the $\Wass_2$ distance (\cf \cite{MRS1,MRS2}), where decongestion is performed using the Laplace operator as we discussed in Remark-\ref{remark:quadraric_case}. The solution of the continuity equation is computed first (prediction step), then it is projected onto the set of admissible densities with respect to $\Wass_2$-Wasserstein distance (correction step). Using our approach, this can be simply solved by changing the functional $\A_h$ to $\A_h(\rho,\Phi)  = \I_{[0,1]}(\rho) + 1/2\Vert\Phi\Vert_{2}^{2}$ and modifying formula \eqref{eq:prox_F} using the fact that 
	\begin{equation}\label{eq:prox_F2}
		\prox_{\frac{\sigma}{2}\Vert .\Vert_{2}^{2}} (\Phi)= \frac{1}{1+\sigma} \Phi.
	\end{equation}
	To observe differences between the two methods, we consider two examples. In the first one (\cf Figure-\ref{fig:comparison_eg1}), the initial density is $\rho_0(\x) = \mathbf{1}_{S_1}(\x)$ with $$S_1 = [0,\frac{1}{2}]\times [0,1]~\mbox{and}~\Gamma_D = \left(\{1\}\times[0,0.4]\right)\cup\left(\{1\}\times[0.9,1]\right).$$.
	\begin{figure}[H]
		\includegraphics[width=1.\textwidth,center]{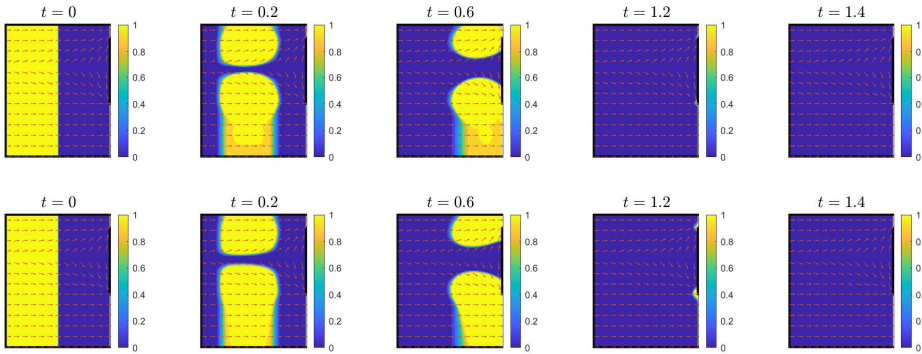}
		\caption{The distribution of crowd at equivalent timesteps with $T = 2$. Top row: result using our approach. Bottom row: result using the Laplacian.}
		\label{fig:comparison_eg1}
	\end{figure}
	\begin{figure}[H]
		\includegraphics[width=0.45\textwidth]{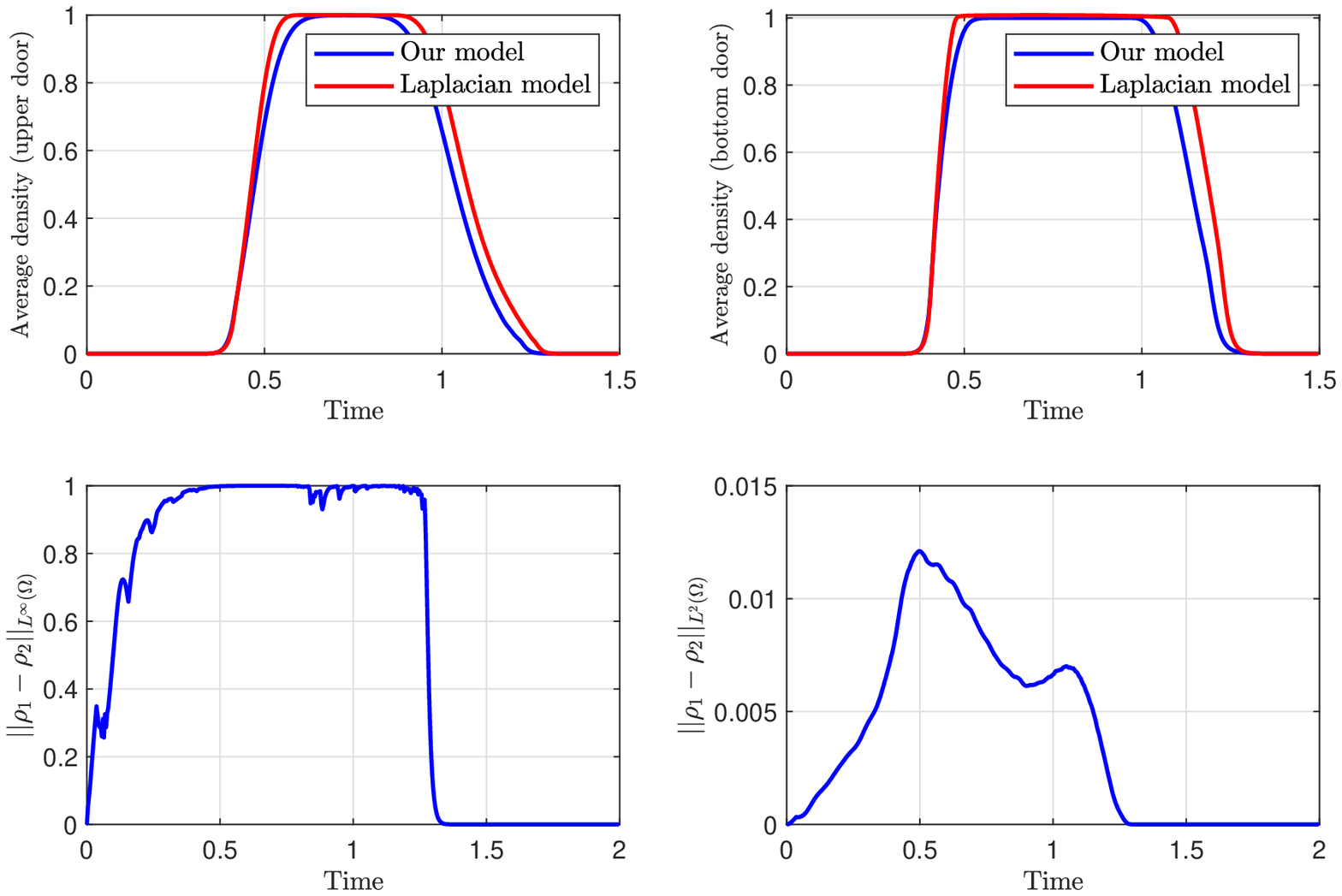}
			\caption{Top: Variation of the average density over time for the two models at the exist doors. Bottom: Variation of $||\rho_1-\rho_2  ||_{L^{\infty}(\Omega)}$ and $||\rho_1-\rho_2  ||_{L^{2}(\Omega)}$  as a function of time for the two rooms case. $\rho_1$ is the solution obtained by our approach and $\rho_2$ the solution obtained by the Laplacian model. }
		\label{fig:norm_comp1}
	\end{figure}

		Now, we consider a domain $\Om = [0,1]^2 = \Om_l \cup \Om_r$ composed of two rooms linked by a bridge in the spirit of \cite{Leclerc&al}, where $\Om_l = [0,0.4] \times [0,1]$ and $\Om_r = [0.6,1] \times [0,1]$.
		The initial density $\rho_0$ is located at the left room and is given by $\rho_0(\x) = \mathbf{1}_{S}(\x)$ with $S = [0,0.4]\times [0,1]$ . The exit is given by the two end points $(1,0)$ and $(1,1)$, that is $\Gamma_D = \{ (1,0),(1,1)\}$.

		\begin{figure}[H]
			\centering
			\includegraphics[width=1.\textwidth,center]{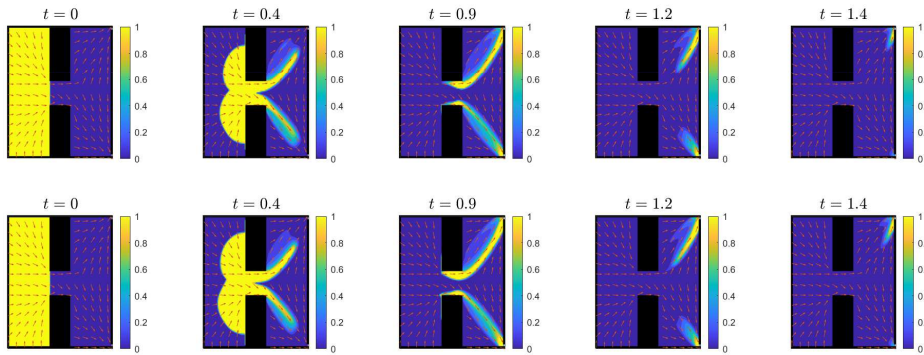}
			\caption{The distribution of the crowd over the domain at equivalent timesteps. Top row: result using our approach. Bottom row: result using the Laplacian.}
			\label{fig:comparison_eg2}
		\end{figure}
		\begin{figure}[H]
			\includegraphics[width=0.45\textwidth]{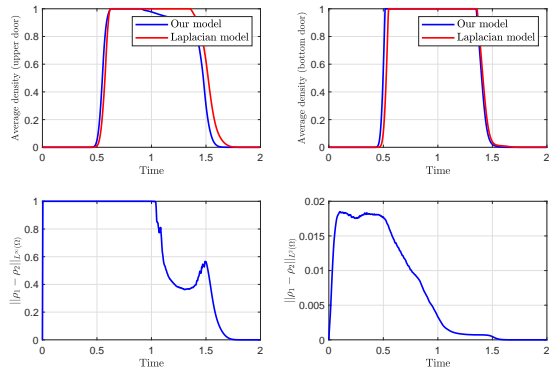}
			\caption{Top: Variation of the average density over time for the two models at the exist doors. Bottom: Variation of $||\rho_1-\rho_2  ||_{L^{\infty}(\Omega)}$ and $||\rho_1-\rho_2  ||_{L^{2}(\Omega)}$  as a function of time for the two rooms case. $\rho_1$ is the solution obtained by our approach and $\rho_2$ the solution obtained by the Laplacian model. }
			\label{fig:norm_comp2}
		\end{figure}
		Figures-\ref{fig:comparison_eg1}-\ref{fig:comparison_eg2} provide a comparison between our method to and one using the Laplace operator in equivalent timesteps. Overall, both models behave similarly except that our model seems to perform faster evacuation. In most of the timesteps examples, it is difficult to visualize differences in of the evolution of the crowd only through the figures. Yet, we can observe this by measuring the $L^{\infty}$ and $L^{2}$ norms of the obtained solutions as well as the variation of the average density over time for the two models at the exist doors. Thanks to Figures-\ref{fig:norm_comp1}-\ref{fig:norm_comp2}, one can clearly notice that our model is faster than the Laplacian model in achieving population evacuation, as the blue curve (our model) remains under the red curve (Laplacian model) over all the time period.

		\subsection{Evacuation with path obstacles} 
		In this section,  we analyse the evacuation process in the presence of in-domain obstacles.  At the microscopic level, it was shown in \cite{alreda}  that pedestrians might be blocked from exiting the room in case where no obstacle is placed in front of the exist. The reason is that pedestrians start to push each other once near to the exist blocking the continuation of the evacuation process. The authors \cite{alreda} have concluded that placing an obstacle just in front of the exist regulates the evacuation and avoids blocking of pedestrians. To observe the effect of placing an obstacle in front in the exist on the fluidity and speed of the evacuation in the macroscopic case, we consider the following example in $\Om = [0,1]^2$ where the obstacle is placed at the region $[0.8,0.9]\times [0.2,0.7]$. The initial density $\rho_0$ is located at the left room and is given by $\rho_0(\x) = \mathbf{1}_{S}(\x)$ with $S = [0,0.5]\times [0,1]$ . The exit is given by $\Gamma_D ={1}\times [0.4,0.6]$.
		
		\begin{figure}[H]
			\centering
			\includegraphics[width=1.\textwidth,center]{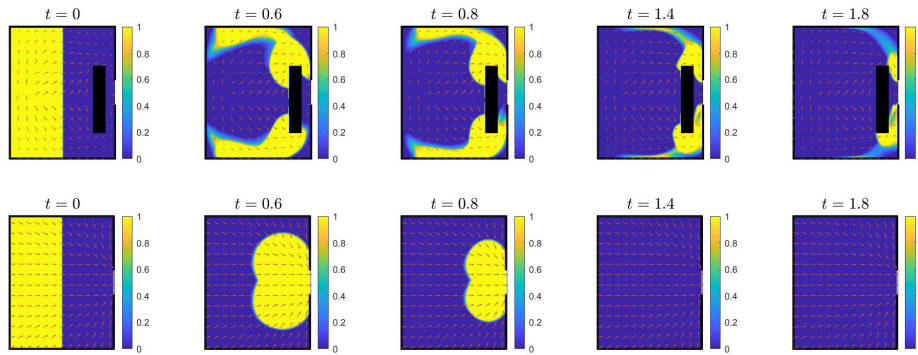}
			\caption{The distribution of the crowd over the domain at equivalent timesteps.}
			\label{fig:comparison_obs}
		\end{figure}
		
		As shown in Figure-\ref{fig:comparison_obs}, we can notice that after $t=1.4$, the room is completely evacuated  in absence of the obstacle in front of the exist. However,  when considering obstacle we can notice that the evacuation is partial and some pedestrian are stuck in the room. In fact, placing an obstacle slowed down the evacuation.\\
		Unlike the microscopic case, adding an obstacle in the macroscopic case have had negative effects on the evacuation process due to the continum model of the density.
		\color{black}
		\section*{Acknowledgment} The work of H.E was partially supported by the ANR grant, reference ANR-20-CE38-0007.

		\begin{appendices}
			
			\section{On the discrete operators}\label{subsection:discrete_op}\hfill\\
			In this section, we recall some  details concerning  the discrete divergence and gradient operators that were used in Section-\ref{section:numerics}. First, let us recall that the space $X = \R^{m\times n}$  is equipped with a scalar product and an associated norm as follows:
			\[
			\langle u,v \rangle = h^2\sum_{i=1}^{m}\sum_{j=1}^{n} u_{i,j} v_{i,j} \quad \mbox{ and } \quad \Vert u\Vert = \sqrt{\langle u, u\rangle},
			\] 
			where $h$ is a given mesh size. Following the definition of the discrete divergence operator given in \eqref{eq:div1}, the discrete gradient  $\nabla_h: X \longrightarrow Y=\R ^{(m+1)\times n} \times \R ^{m\times (n+1)} $ is given by $(\nabla_h u)_{i,j} = \Big((\nabla_h u)^{1}_{i,j},(\nabla_h u)^{2}_{i,j}\Big)$, where 
			\begin{equation}
				\label{eq:grad_h2}
				\begin{aligned}
					(\nabla_{h}u)^1_{i,j}&= ~-^tD^1_p u(i,j) ,~\text{if} \: ((m+\frac{1}{2} )h,j h)  \in \Gamma_D,\\
					(\nabla_{h}u)^1_{i,j}&= ~-^tD^1_mu(i,j) ,  ~\text{if} \: ((m+\frac{1}{2} )h,j h)  \in \Gamma_N,\\
					(\nabla_{h}u)^2_{i,j}&= ~-^tD^2u(i,j).
				\end{aligned}
			\end{equation}
			
			and the matrices $D^{1}_{P}, D^{2}_{m}$, $D^{2}$ are given by
			\[
			D^{1}_{p}=\scriptsize{\begin{pmatrix}
					0 & 1/h &0 &\cdots &&  & 0 \\
					0 & - 1/h &  1/h &0 &\cdots&   &0 \\
					0 &0 & - 1/h &  1/h &0 &\cdots &0 \\
					\vdots &&  \vdots & \ddots &&  \vdots \\
					0 & 0& \cdots & &0& - 1/h &  1/h
			\end{pmatrix}}
			\]
			\[
			D^1_m=\scriptsize{\begin{pmatrix}
					0 & 1/h &0 &\cdots &&  & 0 \\
					0 & - 1/h &  1/h &0 &\cdots&   &0 \\
					0 &0 & - 1/h &  1/h &0 &\cdots &0 \\
					\vdots &&  \vdots & \ddots &&  \vdots \\
					0 & 0& \cdots & &0& - 1/h &  0
			\end{pmatrix}}
			\]
			and
			\[
			D^2=\scriptsize{\begin{pmatrix}
					0 & 1/h &0 &\cdots &&  & 0 \\
					0 & - 1/h &  1/h &0 &\cdots&   &0 \\
					0 &0 & - 1/h &  1/h &0 &\cdots &0 \\
					\vdots &&  \vdots & \ddots &&  \vdots \\
					0 & 0& \cdots & &0& - 1/h &  0
			\end{pmatrix}}.
			\]

			This being said, we check easily that $-\odiv_h$ and $\nabla_h$ are in duality. Moreover, we recall the following

			\begin{proposition}(\cite{c2004,CP}) Under the above-mentioned definitions and notations, one has that
				\begin{itemize}
					\item The adjoint operator of $\nabla_h$ is $\nabla^{*}_h = -\operatorname{div}_h.$
					\item Its norm satisfies: $\Vert \nabla_h \Vert^2 = \Vert \operatorname{div}_h\Vert^2\leq 8/h^2$.
				\end{itemize}
			\end{proposition}
			
			\section{Discretization of the eikonal equation:}\label{subsection:eikonal}\hfill\\	  
			For a self-contained presentation, let us recall our main approach to compute the velocity field $V$ by solving the eikonal equation \eqref{hjk}. As pointed out in \cite{EINsfs} (see also \cite{EINHJ}), the solution $\dist$ of \eqref{hjk} can be obtained by solving
			\begin{equation}\label{eq:distance}
				\max_{u\in W^{1,\infty}(\Omega)}\left\{ \int_{\Omega} u \dd \x: \vert\nabla  u\vert \leq f,~u = 0~\mbox{on}~\Gamma_D\right\} 
			\end{equation}
			which can be written, at a discrete level, as
			\begin{equation}
				\label{P1}
				\min_{u\in X} \A_h(u) + \B_h(\nabla_h u),
			\end{equation}
			where
			\begin{equation}
				\label{functionals:F_G}
				\A_h(u) = \begin{cases} - h^2\sum_{i=1}^{m}\sum_{j=1}^{n} u_{i,j} &\text{ if } u_{i,j}=0 \:\: \forall (i,j) \in D_d\\
					+\infty & \text{ otherwise }\end{cases},~\mbox{and}~\B_h=\mathbb{I}_{B(0,f)},
			\end{equation}
			where $D_d=\left\{(i, j): (ih,jh)\in\Gamma_D\right\}$ the indexes whose spatial positions belong to $\Gamma_D$ and $B(0,f)$ is the unit ball of radius $f$. Then we apply Algorithm-\ref{alg:pd1} with the functionals $\A_h$ and $\B_h$ above.

			
			\section{Duality results}\label{subsection:duality}\hfill\\	  
			The idea for the proof of Lemma \ref{phj} goes back to   \cite[Theorem 3.10]{EINHJ}.   The aim is to define a convex and l.s.c functional $\mathcal H \: :\:  \mathcal M_b(\overline \Omega) \mapsto  ]-\infty,\infty]$  such that 
			$$\mathcal H(0)=  \inf_{\tau \Phi\in \F^q(\tilde \rho-\rho)}\left \{  \tau \int_\Omega \kk(x)\: \vert \Phi(x)\vert\: \dd x -\tau \int_{\Gamma_D} g(x)\: \Phi\cdot \bnu \:     \dd x   \right\}   $$  
			and  
			$$ \sup_{p\in \mathcal C(\overline \Omega)} -\mathcal H^{*}(p)  = 	\max_{p\in \G_\kk }\left \{   \int (\tilde \rho-\rho)\: p\:  \dd x   - \tau\:  \int_{\Gamma_N} p\: \eta \: \dd x \right\} .$$ 
			Then conclude by classical duality results 
			\begin{equation} \label{perturbduality}
				\mathcal H(0) = \mathcal H^{**}(0) = \sup_{ p\in \mathcal C(\overline \Omega)} -\mathcal H^{*}(p) .
			\end{equation} 
			One sees that in order to built $\mathcal H$, we need to use vector fields   whose divergences are Radon measures (\cf \cite{Chen&Frid}).    
			Thanks to \cite{Chen&Frid},  we know  that  for any   $\Phi\in L^1(\Omega)^N$  such that $ \dive\Phi =: \mu_\Phi\in \mathcal M_b(\Omega),$ in the sense of $\D'(\Omega),$      the normal trace of $\Phi$   is well defined on the boundary of $\Omega$.   Indeed, for such $\Phi,$ we have $\Phi \cdot \bnu \: :\: \Lip(\partial\Omega )\to \R$ is  continuous linear functional and satisfies 
			\begin{equation}\label{tracemes}
				\langle \Phi \cdot \bnu ,\xi_{\mid\partial \Omega}\rangle =  \int_{  \Omega} \xi\: \dd\mu_\Phi  + \int_{\Omega} \Phi\cdot\nabla \xi\:\dd x ,\quad \hbox{ for any }  \xi\in \mathcal C^1 (\overline{\Omega})  .
			\end{equation}
			We will denote again $ \langle \Phi \cdot \bnu ,\xi_{\mid\partial \Omega}\rangle =: \int_{\partial \Omega} \Phi \cdot \bnu  \: \xi\: \dd x.$ 
			As for $H_{\hbox{div}}$ vector valued field, it is possible to define the restriction of the normal trace of $\Phi$ on    $\Gamma_D$ by working with  Lipschitz  continuous     test functions which vanishes on $\Gamma_N. $
			For any $g \in \Lip(\Gamma_D)$ such that there exists $\tilde g\in \mathcal C^1(\overline \Omega),$ satisfying 
			$$\tilde g = g \hbox{ on }\Gamma_D\hbox{ and }\tilde g =0\hbox{ on }\Gamma_N.$$ 
			For such $g,$ we'll define  
			$\int_{\Gamma_D} \Phi \cdot \bnu  \: g\: \dd x :=  \langle \Phi \cdot \bnu ,\tilde g_{\mid\partial \Omega}\rangle $, for any $\Phi\in L^1(\Omega)^N$; \ie 
			\begin{equation}\label{traceformula}
				\int_{\Gamma_D} \Phi \cdot \bnu  \: g\: \dd x =   \int_{  \Omega} \tilde g\:\dd\mu_\Phi    + \int_{\Omega} \nabla\tilde g  \cdot \Phi \: \dd x. 
			\end{equation} 
			So, for any $\mu\in \mathcal M_b(\overline \Omega),$ we can define 
			$$\F (\mu):=\Big\{  \Phi\in L^{1}(\Omega)^N \: :\:  - \dive(\Phi)  = \mu\res \Omega     \hbox{ in }      \Omega    \hbox{ and }    \Phi\cdot \bnu =\eta+\mu\res \Gamma_N \hbox{ on } \Gamma_N    \Big\}.  $$ 
			Here, the condition  $- \dive(\Phi)  = \mu\res \Omega$  in $ \Omega$ and  $  \Phi\cdot \bnu =\eta+\mu\res \Gamma_N$  on  $ \Gamma_N,$ needs to be understood in the sense   
			\begin{equation}
				\int_\Omega \Phi\cdot \nabla \xi \: \dd x = \int_{\overline \Omega }  \xi\: \dd\mu  + \int_{\Gamma_N} \eta\: \xi\: \dd x  , \quad \hbox{ for any }\xi\in \Lip(\Omega)\hbox{ s.t. }\xi_{\mid\Gamma_D}=0.
			\end{equation}

			\begin{proposition}
				\label{phj1}
				For any $\mu\in \mathcal M_b(\overline \Omega)$,  we have 
				\begin{equation}\label{std11}
					\begin{array}{l}
						\inf_{\tau\Phi\in\F  (\mu)}\left \{ \tau \int_\Omega \kk(x)\: \vert \Phi(x)\vert\: \dd x -\tau\int_{\Gamma_D} g(x)\: \Phi\cdot \bnu \:     \dd x   \right\} \\  \\
						\hspace*{1cm}= 	\max_{p\in \G_\kk }\left \{   \int_{\overline \Omega}  p\:  \dd\mu   - \tau\:  \int_{\Gamma_N} p\: \eta \: \dd x \right\}   . 
					\end{array} 
				\end{equation} 
			\end{proposition}
			\begin{proof}
				We consider on $ \mathcal M_b(\overline \Omega)$  the following functional  $\mathcal H\: :\:  \mathcal M_b(\overline \Omega) \mapsto  ]-\infty,\infty]$ defined by 
				\[
				\mathcal H(h) = \inf_{\tau \phi\in    \F  (\mu+h)  } \left\{\tau \int_{{\Omega}}  \kk(x)\: \vert \Phi(x)\vert\dd x+    \int_{\Gamma_D} g\: \dd h -\tau   \int_{\Gamma_D}\Phi\cdot \bnu \: g\: \dd x  \right\}, 
				\]
				for any $h\in \mathcal M_b(\overline \Omega).$ 
				Then $	\mathcal H$  is convex and l.s.c. 
				\paragraph{\textbf{Convexity.}} Indeed, take $h_1,h_2 \in\mathcal{M}_{b}(\overline \Omega)$ and set $h := th_1 + (1-t)h_2$ for $t\in [0,1]$. Let $\Phi_{1,n},\Phi_{2,n}\in L^1(\Omega)^N  $ be two minimizing sequences of fluxes corresponding to $h_{1}$ and $h_{2}$ respectively, \ie $\tau\Phi_{1,n}\in \F  (\mu+h_1) $ and $\tau\Phi_{2,n}\in \F  (\mu+h_2) $  such that 
				$$	\mathcal H(h_{i}) =\tau \lim_{n} \int_{{\Omega}} \kk(x)\vert\Phi_{i,n}(x)\vert \dd x+\int_{\Gamma_D} g\: \dd h_i  - \tau \int_{\Gamma_D}  \Phi_{i,n}\cdot \bnu  \: g\: \dd x ~~\text{ for} ~~i=1,2.$$
				Set $\Phi_{n} = t\Phi_{1,n} + (1-t)\Phi_{2,n}$. We clearly see that $\phi_n$ are admissible for $h$ and
				\begin{align*}
					\mathcal H(h) &\leq \tau  \lim_{n} \int_{{\Omega}} \kk(x)\vert\Phi_{n}(x)\vert \dd x +\int_{\Gamma_D} g\: \dd h -  \tau  \int_{\Gamma_D}  \Phi_{n}\cdot \bnu  \: g\: dx \\
					& = \tau  \lim_{n}\int_{{\Omega}} \kk(x)\vert t\Phi_{1,n} + (1-t)\Phi_{2,n}\vert \dd x +\int_{\Gamma_D} g\: \dd h -  \tau  \int_{\Gamma_D} (t\Phi_{1,n}  + (1-t)\Phi_{2,n})\cdot   \bnu \: g\: \dd x  \\
					&\leq    \lim_{n} t\Big(\tau  \int_{{\Omega}} \kk(x)\vert\Phi_{1,n}\vert\dd x  +\int_{\Gamma_D} g\: \dd h -  \tau  \int_{\Gamma_D}   \Phi_{1,n}\cdot \bnu  \: g\: \dd x  \Big) \\  
					&  \hspace*{3cm}+    (1-t)\Big(\tau \int_{{\Omega}} \kk(x)\vert\Phi_{2,n}\vert\dd x+\int_{\Gamma_D} g\: \dd h   -\tau  \int_{\Gamma_D}\Phi\cdot \bnu \: g\: \dd x  \Big)\\
					&\leq t 	\mathcal H(h_1) + (1-t)	\mathcal H(h_2)
				\end{align*}
				and this proves convexity. 
				\paragraph{\textbf{Lower semicontinuity.}} Take a sequence $h_n\rightharpoonup h$ in $\mathcal M_b(\overline \Omega)$. For every $n\in\nn$, we consider a sequence $\tau (\Phi_{n}^{k})_{k\in\mathbb{N}}$ of     $   \F   (\mu+h_n)$  such that  
				$$	\mathcal H(h_n) = \tau  \lim_{k\to\infty} \int_{{\Omega}}	\kk(x)\vert\Phi_{n}^{k}(x)\vert \dd x  +  \int_{\Gamma_D} g\: \dd h_n -  \tau  \int_{\Gamma_D}   \Phi_{n}^{k}\cdot \bnu \: g\: \dd x.$$ 
				We may find some $\psi_{n}\in L^{1}({\Omega})^N$ satisfying 	
				\begin{equation}\label{psin}
					\tau\: 	\int_\Omega \psi_n\cdot \nabla \xi \: \dd x = \int_{ \overline  \Omega }(h-h_n)\: \xi\: \dd x, \quad \hbox{ for any }\xi\in \Lip(\Omega)\cap W^{1,2}_{\Gamma_D } (\Omega).
				\end{equation}
				and such that    $\Vert\psi_n\Vert_{L^1}\to 0$ and $ \langle \psi_{n}\cdot \bnu ,g\rangle\to 0.$  In fact, since  $W^{1,s}_{\Gamma_D}(\Omega) \hookrightarrow \mathcal C(\overline \Omega),$ for $s> N,$   one sees that the optimization problem  
				$$\min_{z\in W^{1,s}_{\Gamma_D}({\Omega})}  \left\{ \frac{1}{s}  \int_\Omega \vert \nabla z\vert^s\: \dd x -\int_{\overline \Omega} z\dd (h-h_n) \right\} $$ 
				has a unique solution that we denote   $u_n.$  We consider   $\psi_n := \vert \nabla u_n \vert^{s-2} \nabla u_n.$ It is clear that $\psi_n\in L^{s'} ( \Omega ),$ and then in $L^1(  {\Omega}).$   Moreover, using standard techniques of calculus of variation, we see that $ \psi_n $  satisfies \eqref{psin}.  
				Clearly  $u_n$ is bounded in $W^{1,s}_{\Gamma_D}( \Omega),$ so that by taking a subsequence if necessary, we have  $u_n\rightharpoonup u$ in  $W^{1,s}( \Omega) $ and uniformly in $\overline {\Omega}.$   This implies that 
				$$ \tau   \int_{ \Omega }\vert  \psi_n\vert ^{s'}\: \dd x =   \tau  \int_{ \Omega } \vert\nabla u_n\vert^{s} \dd x =  \int_{\overline \Omega} u_n\: \dd (h-h_n)\quad \underset{n\to\infty} {\longrightarrow } 0$$
				and then  $\vert  \psi_n\vert  \underset{n\to\infty} {\longrightarrow } 0 $ in $L^1({\Omega}).$  In addition,    thanks to \eqref{traceformula},   we have  
				$$  \tau \langle\psi_{n}\cdot \bnu  ,g\rangle =  \tau   \int_{  {\Omega}}  \psi_n\cdot \nabla \tilde g \: \dd x -  \int_{\Omega}  \tilde g \: \dd (h-h_n)  \quad \underset{n\to\infty} {\longrightarrow } 0 .$$  
				This being said, we clearly have $- \mathrm{div}(\Phi_{n}^{k} +\psi_n) = \mu+  h\res \Omega$ in  $ \Omega  $ and  $(\Phi_{n}^{k} +\psi_n)\cdot \bnu=\eta+   h\res \Gamma_N$ on $\Gamma_N$; \ie $  \tau  (\Phi_{n}^{k}+\psi_{n})\in \F  (\mu+h)$.  By semicontinuity of the integral, we have
				\begin{align*}
					\mathcal H(h) &\leq \tau  \int_{{\Omega}} \kk(x)\vert(\Phi_{n}^{k} + \psi_{n})(x)\vert \dd x  +  \int_{\Gamma_D} g\: \dd h -  \tau  \int_{\Gamma_D}   (\Phi_{n}^{k} + \psi_{n})\cdot \bnu \: g\: \dd x\\
					&\leq  \tau   \int_{{\Omega}}\kk(x)\vert\Phi_{n}^{k}(x)\vert\dd x + \int_{\Gamma_D} g\: \dd h_n -   \tau  \int_{\Gamma_D} \Phi_{n}^{k}\cdot \bnu  \: g\: \dd x  \\& \quad+  \tau  \int_{{\Omega}} \kk(x)\vert\psi_{n}(x)\vert\dd x  +  \int_{\Gamma_D} g\: \dd (h -h_n) -   \tau  \int_{\Gamma_D} \psi_{n}\cdot \bnu  \: g\: \dd x .
				\end{align*}
				Letting $k\to\infty$ we get
				\[
				\mathcal H(h) \leq	\mathcal H (h_n) +  \tau  \int_{{\Omega}} \kk(x)\vert\psi_{n}(x)\vert\dd x+\int_{\Gamma_D} g\: \dd(h-h_n)   -\tau   \int_{\Gamma_D} \psi_{n}\cdot \bnu  \: g\: \dd x.
				\]
				Now, letting $n\to\infty,$ and using the fact that $ \psi_{n} \to 0$ in $L^1({\Omega})^N,$ and $h_n \rightharpoonup h$ in $\mathcal{M}_{b}(\overline \Omega),$ as $n\to\infty$, we obtain the lower semicontinuity, \ie
				\[
				\mathcal H(h) \leq \liminf_{n} 	\mathcal H(h_n).
				\]
				Next let us compute $	\mathcal H^*$. For any  $p\in \CC(\overline \Omega),$ we have 
				\begin{eqnarray*}
					\mathcal H^{*}(p) &=& \sup_{h\in\mathcal{M}_{b}(\overline \Omega)} \left\{ \int_{\overline \Omega} p\dd h - \mathcal H(h)\right\} \\
					&=& \sup_{ h\in\mathcal{M}_{b}(\overline \Omega),\: \tau  \Phi\in\F  (\mu+h)}    \left\{ \int_{\overline \Omega} p \dd h - \tau   \int_{{\Omega}} \kk(x)\vert\Phi(x)\vert\dd x- \int_{\Gamma_D} g\: \dd h +  \tau  \int_{\Gamma_D}    \Phi\cdot \bnu \: g\: \dd x 
					\right\} \\
					&=& I_1(p) + I_2(p),
				\end{eqnarray*}
				where $I_1(p) := -\int_{\overline \Omega \setminus \Gamma_D} p\: \dd \mu   $ and $I_2(p)$  is given by 
				$$ \sup_{h\in\mathcal{M}_{b}(\overline \Omega),\: \tau  \Phi\in\F  (\mu+h)}    \left \{ 
				\int_{\overline \Omega\setminus \Gamma_D} p \dd  (\mu+h) -  \tau  \int_{{\Omega}} \kk(x)\: \vert \Phi(x)\vert \dd x    +  \int_{\Gamma_D} (p-g)\: \dd h   +  \tau  \int_{\Gamma_D}  \Phi\cdot \bnu \: g\: \dd x  \right\}. $$
				Using Lemma \ref{Ltech} below, we deduce that, for any $u\in \Lip(\Omega),$ we have  
				$$	\mathcal H^{*}(p) =  \left\{   \begin{array}{ll} 
					-\int_{\overline \Omega \setminus \Gamma_D} p\: \dd\mu    \quad &  ~~\text{if}~~u\in \G_\kk \\  \\  
					\infty & ~~\text{otherwise}.\end{array}\right. 
				$$ 
				Finally, using \eqref{perturbduality}  we deduce the result. 
			\end{proof}
			
			\begin{lemma}\label{Ltech}
				Let $p\in \Lip(\Omega),$ we have 	
				$$\sup_{\underset{ \tau  \Phi\in\F  (\mu+h)}{h\in\mathcal{M}_{b}(\overline \Omega)}}    \left\{\int_{\overline \Omega\setminus \Gamma_D} p \dd(\mu +h) -  \tau  \int_{{\Omega}} \kk(x) \vert \Phi(x)\vert \dd x    +  \int_{\Gamma_D}(p- g)\: \dd h   +  \tau    \int_{\Gamma_D}  \Phi\cdot \bnu \: g\: \dd x  \right\}
				=\left\{   \begin{array}{ll} 
					0  \quad &  ~~if~~p\in \G_\kk\\  \\  
					\infty & ~~\hbox{otherwise}.\end{array}\right.$$	 
			\end{lemma}
			\begin{proof} Take $p$ as a test function in the divergence constraint $- \tau \mathrm{div}(\Phi) = \mu +h~\mbox{in}~\mathcal{D}^{'}({\overline \Omega\setminus  \Gamma_D })$, we get 
				\begin{equation*}
					\begin{aligned} I(h, \Phi)&:=\int_{\overline \Omega\setminus \Gamma_D} p  \:  \dd(\mu +h) - \tau  \int_{{\Omega}} \kk(x)\vert\Phi(x)\vert\:  \dd x    +  \int_{\Gamma_D}(p- g)\: \dd h   +  \tau \int_{\Gamma_D}  \Phi\cdot \bnu \: g\: \dd x  \\
						&=\tau \int_{\Omega} \nabla p \cdot \Phi  \: \dd x- \tau \int_{{\Omega}} \kk(x)\vert \Phi(x)\vert  \:  \dd x    + \int_{\Gamma_D} (p- g)\: \dd h   +\tau \int_{\Gamma_D}  \Phi\cdot \bnu \: (g-p)\: \dd x .
					\end{aligned}
				\end{equation*}
				It is clear that for any  $   p\in \G_\kk,$ we have $I(h, \Phi)\le 0$, and by taking $h\equiv-\mu $ and $\Phi\equiv 0$ we obtain 
				$\sup I(h, \Phi)=0. $   For the case where  $p\not\equiv g$ on $\Gamma_D $; \ie  $p (x_0)\neq g (x_0)$  for some $x_0\in \Gamma_D,$    one can work with   $h_n=n \:  \hbox{Sign} (p(x_0)-g(x_0))\delta_{x_0}$ for $n\in\nn$, where  $\delta_{x_0}$ is Dirac mass at $x_0$,  and fix any $\Phi_0\in\F  (\mu+h)$ such that $-\hbox{div }\Phi_0 = \mu $ in $\F  (\mu+h)$,  to see that   $I(h_n, \Phi)\longrightarrow \infty, $ as  ${n\to\infty}  .$  Now, 
				for the remaining case, \ie $p=g$ on $\Gamma_D$ and  $\vert \nabla p\vert > \kk $    on a   subset $A\subset {\Omega}$ such that $\vert A\vert \neq 0,$        we consider 
				$  \Phi_{n\epsilon}=n\: (\nabla p\:  \chi_A )*  \eta_\epsilon  ,$ where $\eta_\epsilon$ is a sequence of    mollifiers.  It is clear that there exists $h\in \mathcal M_b(\overline \Omega),$ such that $-\hbox{div }  \Phi_{n\epsilon} =\mu +h$ in $\D'(\overline \Omega\setminus \Gamma_D).$  Moreover, for any $n,$ we have 
				$$\sup I(h, \phi) \ge  \tau \int_{\Omega} \Phi_{n\epsilon} \cdot \nabla u  - \tau \int_{\Omega}\kk(x)\vert \nabla p (x)\vert \: \dd x. $$
				Letting $\epsilon \to 0,$ we get 
				$$\sup I(h, \Phi)\ge n \tau \underbrace{ \int_{A} \left(   \vert\nabla p(x)\vert^2  -   \kk(x)\: \vert \nabla p(x)\vert \right) \dd x}_{>0}   \longrightarrow \infty, \quad \hbox{ as }{n\to\infty}  .$$
				This concludes the proof. 
			\end{proof}

			\bigskip 
			\begin{proof}[Proof of Lemma \ref{phj}] Now, the proof is a simple consequence of Proposition \ref{phj1}.
			\end{proof}

		\end{appendices}
		
		
		\bibliographystyle{abbrv}
		\bibliography{cm}

\begin{thebibliography}{10}

\bibitem{AgCaIg}
M.~Agueh, G.~Carlier, and N.~Igbida.
\newblock On the minimizing movement with the 1-{W}asserstein distance.
\newblock {\em ESAIM Control Optim. Calc. Var.}, 24(4):1415--1427, 2018.

\bibitem{Beckmann}
M.~Beckmann.
\newblock A continuous model of transportation.
\newblock {\em Econometrica}, 20:643--660, 1952.

\bibitem{Bellomo&Dogbe}
N.~Bellomo and C.~Dogb{\'e}.
\newblock On the modelling crowd dynamics from scaling to hyperbolic
  macroscopic models.
\newblock {\em Math. Models Methods Appl. Sci.}, 18:1317--1345, 2008.

\bibitem{Bord}
T.~Bord.
\newblock Highway capacity manual, 204 {TRB}.

\bibitem{Borwein&Zhuang}
J.~M. Borwein and D.~Zhuang.
\newblock On {F}an's minimax theorem.
\newblock {\em Math. Programming}, 34(2):232--234, 1986.

\bibitem{Bouchitte&al}
G.~Bouchitte, G.~Buttazzo, and P.~Seppecher.
\newblock Energies with respect to a measure and applications to low
  dimensional structures.
\newblock {\em ArXiv preprint arXiv: 2105.00182}, 5, 1997.

\bibitem{c2004}
A.~{Chambolle}.
\newblock {An algorithm for total variation minimization and applications.}
\newblock {\em {J. Math. Imaging Vis.}}, 20(1-2):89--97, 2004.

\bibitem{CP}
A.~{Chambolle} and T.~{Pock}.
\newblock {A first-order primal-dual algorithm for convex problems with
  applications to imaging.}
\newblock {\em {J. Math. Imaging Vis.}}, 40(1):120--145, 2011.

\bibitem{Chen&Frid}
G.-Q. Chen and H.~Frid.
\newblock Divergence-measure fields and hyperbolic conservation laws.
\newblock {\em Arch. Ration. Mech. Anal.}, 147(2):89--118, 1999.

\bibitem{Colombo&Rosini}
R.~M. Colombo and M.~D. Rosini.
\newblock Pedestrian flows and non-classical shocks.
\newblock {\em Math. Methods Appl. Sci.}, 28(13):1553--1567, 2005.

\bibitem{Coscia&Canavesio}
V.~Coscia and C.~Canavesio.
\newblock First-order macroscopic modelling of human crowd dynamics.
\newblock {\em Math. Models Methods Appl. Sci.}, 18(suppl.):1217--1247, 2008.

\bibitem{Noemi&Markus}
N.~David and M.~Schmidtchen.
\newblock On the incompressible limit for a tumour growth model incorporating
  convective effects, 2021.
\newblock arXiv preprint:~\url{https://arxiv.org/abs/2103.02564}.

\bibitem{Degond&al}
P.~Degond, L.~Navoret, R.~Bon, and D.~Sanchez.
\newblock Congestion in a macroscopic model of self-driven particles modeling
  gregariousness.
\newblock {\em J. Stat. Phys.}, 138(1-3):85--125, 2010.

\bibitem{DM}
S.~Di~Marino and A.~R. M\'{e}sz\'{a}ros.
\newblock Uniqueness issues for evolution equations with density constraints.
\newblock {\em Math. Models Methods Appl. Sci.}, 26(9):1761--1783, 2016.

\bibitem{Dogbe}
C.~Dogbé.
\newblock On the numerical solutions of second order macroscopic models of
  pedestrian flows.
\newblock {\em Computers \& Mathematics with Applications}, 56(7):1884--1898,
  2008.

\bibitem{IgDu}
S.~Dumont and N.~Igbida.
\newblock On a dual formulation for the growing sandpile problem.
\newblock {\em European J. Appl. Math.}, 20(2):169--185, 2009.

\bibitem{Dweik}
S.~Dweik.
\newblock {$W^{1,p}$} regularity on the solution of the {BV} least gradient
  problem with {D}irichlet condition on a part of the boundary.
\newblock {\em Nonlinear Anal.}, 223:Paper No. 113012, 18, 2022.

\bibitem{Ekeland&Temam}
I.~Ekeland and R.~Temam.
\newblock {\em Convex analysis and variational problems}.
\newblock Studies in Mathematics and its Applications, Vol. 1. North-Holland
  Publishing Co., Amsterdam-Oxford; American Elsevier Publishing Co., Inc., New
  York, 1976.
\newblock Translated from the French.

\bibitem{EINHJ}
H.~Ennaji, N.~Igbida, and V.~T. Nguyen.
\newblock Augmented {L}agrangian methods for degenerate {H}amilton-{J}acobi
  equations.
\newblock {\em Calc. Var. Partial Differential Equations}, 60(6):Paper No. 238,
  28, 2021.

\bibitem{EINfinsler}
H.~Ennaji, N.~Igbida, and V.~T. Nguyen.
\newblock Beckmann-type problem for degenerate {H}amilton-{J}acobi equations.
\newblock {\em Quart. Appl. Math.}, 80(2):201--220, 2022.

\bibitem{EINsfs}
H.~Ennaji, N.~Igbida, and V.~T. Nguyen.
\newblock Continuous {L}ambertian shape from shading: a primal-dual algorithm.
\newblock {\em ESAIM Math. Model. Numer. Anal.}, 56(2):485--504, 2022.

\bibitem{EvReza}
L.~C. Evans and F.~Rezakhanlou.
\newblock A stochastic model for growing sandpiles and its continuum limit.
\newblock {\em Comm. Math. Phys.}, 197(2):325--345, 1998.

\bibitem{Helbing1}
D.~Helbing.
\newblock A mathematical model for the behavior of pedestrians.
\newblock {\em Systems Research and Behavioral Science}, 36:298--310, 1991.

\bibitem{Helbing2}
D.~Helbing, P.~Molnar, and F.~Schweitzer.
\newblock Computer simulations of pedestrian dynamics and trail formation.
\newblock 01 1994.

\bibitem{Hughes1}
R.~L. Hughes.
\newblock The flow of human crowds.
\newblock In {\em Annual review of fluid mechanics, {V}ol. 35}, volume~35 of
  {\em Annu. Rev. Fluid Mech.}, pages 169--182. Annual Reviews, Palo Alto, CA,
  2003.

\bibitem{Hughes2}
R.~L. Hughes.
\newblock The flow of human crowds.
\newblock {\em Annual Review of Fluid Mechanics}, 35:169--182, 2003.

\bibitem{IgCrowd}
N.~Igbida.
\newblock New variant of cross-diffusion system.

\bibitem{Igshuniq}
N.~Igbida.
\newblock $\sc{L}^1-$theory for reaction-diffusion hele-shaw flow with linear
  drift.
\newblock Accepted in Math. Meth. Appl. Sciences.:~
  \url{https://arxiv.org/abs/2105.00182}.

\bibitem{Igstoch}
N.~Igbida.
\newblock Back on stochastic model for sandpile.
\newblock In {\em Recent developments in nonlinear analysis}, pages 266--277.
  World Sci. Publ., Hackensack, NJ, 2010.

\bibitem{Leclerc&al}
H.~Leclerc, Q.~M\'{e}rigot, F.~Santambrogio, and F.~Stra.
\newblock Lagrangian discretization of crowd motion and linear diffusion.
\newblock {\em SIAM J. Numer. Anal.}, 58(4):2093--2118, 2020.

\bibitem{MRS1}
B.~Maury, A.~Roudneff-Chupin, and F.~Santambrogio.
\newblock A macroscopic crowd motion model of gradient flow type.
\newblock {\em Math. Models Methods Appl. Sci.}, 20(10):1787--1821, 2010.

\bibitem{MRS2}
B.~Maury, A.~Roudneff-Chupin, and F.~Santambrogio.
\newblock Congestion-driven dendritic growth.
\newblock {\em Discrete Contin. Dyn. Syst.}, 34(4):1575--1604, 2014.

\bibitem{MRSV}
B.~Maury, A.~Roudneff-Chupin, F.~Santambrogio, and J.~Venel.
\newblock Handling congestion in crowd motion modeling.
\newblock {\em Netw. Heterog. Media}, 6(3):485--519, 2011.

\bibitem{MS}
A.~R. M\'{e}sz\'{a}ros and F.~Santambrogio.
\newblock Advection-diffusion equations with density constraints.
\newblock {\em Anal. PDE}, 9(3):615--644, 2016.

\bibitem{Piccoli&Tosin1}
B.~Piccoli and A.~Tosin.
\newblock Pedestrian flows in bounded domains with obstacles.
\newblock {\em Contin. Mech. Thermodyn.}, 21(2):85--107, 2009.

\bibitem{Piccoli&Tosin2}
B.~Piccoli and A.~Tosin.
\newblock Time-evolving measures and macroscopic modeling of pedestrian flow.
\newblock {\em Arch. Ration. Mech. Anal.}, 199(3):707--738, 2011.

\bibitem{alreda}
F.~A. Reda.
\newblock {\em {Crowd motion modelisation under some constraints}}.
\newblock Theses, {Universit{\'e} Paris Saclay (COmUE)}, Sept. 2017.

\bibitem{Sanatambrogio-duality}
F.~Santambrogio.
\newblock Regularity via duality in calculus of variations and degenerate
  elliptic {PDE}s.
\newblock {\em J. Math. Anal. Appl.}, 457(2):1649--1674, 2018.

\end{thebibliography}

	\end{document}